\documentclass[10pt, reqno]{amsart}
\usepackage{latexsym}
\usepackage{amssymb}
\usepackage{amsmath}
\usepackage{amscd}
\usepackage{amsthm}
\usepackage{graphicx}
\usepackage{color}
\usepackage{bm}
\newcommand\dela[1]{}
\allowdisplaybreaks

\def\XXint#1#2#3{{\setbox0=\hbox{$#1{#2#3}{\int}$ }
\vcenter{\hbox{$#2#3$ }}\kern-.6\wd0}}

\newtheorem{theorem}{Theorem}
\numberwithin{equation}{section}
\numberwithin{theorem}{section}
\newtheorem{lemma}[theorem]{Lemma}

\newtheorem{remark}[theorem]{Remark}

\numberwithin{equation}{section}

\newcommand{\wt}{\widetilde}
\newcommand{\e}{\varepsilon}

\def\P{\mathbb{ P}}
\def\E{\mathbb{E}}
\def\P1{\mathbb{P}^{1}}

\def\o{\overline}
\begin{document}

%\showkeys

\title{Stochastic homogenization \\
of multicontinuum heterogeneous flows}

\author[H. Bessaih]{Hakima Bessaih}
\address{University of Wyoming, Department of Mathematics and Statistics, Dept. 3036, 1000
East University Avenue, Laramie WY 82071, United States}
\email{ bessaih@uwyo.edu}

%\author[Y. Efendiev]{ Yalchin Efendiev}
%\address{Department of Mathematics \& ISC, 
%Texas A\&M University, 
%Numerical Porous Media SRI Center, CEMSE Division, 
%King Abdullah University of Science and Technology, 
%Thuwal 23955-6900, 
%Kingdom of Saudi Arabia}
%\email{efendiev@math.tamu.edu}

\author[R. F. Maris]{Razvan Florian Maris}
\address{Alexandru Ioan Cuza University of Iasi, Faculty of Economics and Business Administration, Bd. Carol I, 22, 
Iasi 700505, 
Romania}
\email{florian.maris@feaa.uaic.ro}
%\today
\maketitle

{\footnotesize
\begin{center}

\end{center}
}
\begin{abstract}
We consider a multicontinuum model in porous media applications,
which is described as a system of coupled flow equations.
The coupling between different continua depends on many factors
and its modeling is important for porous media applications. 
The coefficients depend 
on particle deposition that is described in term of a stochastic process solution of an SDE. The stochastic process is considered to be faster than the flow motion and we introduce time-space scales to model the problem.  Our goal is to pass to the limit in time and space and to find an associated averaged system.  This is an averaging-homogenization problem, where the averages are computed in terms of the invariant measure associated to the fast motion and the spatial variable. We use the techniques developed in our previous paper \cite{BEM_19}
to model the interactions between the continua and derive the averaged model
problem that can be used in many applications.
\end{abstract}
%\\
%\\
%{\bf Mathematics Subject Classification 2000}: Primary 60H30, 76S05; Secondary  76M35. 

\maketitle

\section{Introduction and formulation of the problem}
\label{sec1}

\subsection{Motivation. }
\label{subs11}

Multicontinua models arise in many applications \cite{barenblatt1960basic}, 
which
include porous media, material sciences, and so on. The main idea
of multicontinua model is to prescribe multiple effective properties
in each macroscale point. These approaches are important in many real-world
applications, where classical homogenization fails. A typical example
is flow in fractured media. One typically cannot upscale fracture and
surrounding media, called matrix. For this reason, different effective
parameters are prescribed in each coarse-grid block. In applications
related to shale gas, one uses multiple continua to describe 
organic and inorganic matter besides fractures \cite{akkutlu2012multiscale,dpgmsfem2017,pororelgmsfem2018,akkutlu2015multiscale}.

Multicontinua approaches have rigorously been justified recently using
constraint energy minimizing Generalized Multiscale Finite Element Method
\cite{CEL} and nonlocal multicontinua approaches \cite{VCELW}. It turns
out that these models can be used for general upscaling and multicontinua
and
can be employed to separate different regions using spectral decompositions.
The spectral decomposition allows identifying each continua and
local basis functions allow coupling these continua with each others.
These approaches can also be applied to nonlinear problems
\cite{nonlinear_NLMC}.

In multicontinua approaches, the interaction between different continua
is challenging to model. This modeling requires complex local solves and
can depend on the solution itself \cite{wheeler}. This dependence can be a 
result of nonlinear modeling or due to particle deposition due to pore-scale
modeling, which occur in porous media applications.
The latter can modify the interaction between the continua, 
which will depend on the solution. The modeling
the interaction coefficients in the presence of particle deposition
introduces additional challenges and their modeling requires
some type of stochastic modeling.
Our goal is to study one such model, 
which is sufficiently general and can be generalized for more complex cases.
Currently, we can only rigorously analyze the proposed model, which
we plan to generalize later.

%In the paper, we derive macroscopic model assuming that the particle dynamics at microscale
%occurs in a much faster time scale compared to the flow. This is typical 
%in these applications due to the particle dynamics and their interaction. %
%We derive
%a macroscopic model where the new upscaled permeability is computed
%using spatial microscale variations of the permeability and fast dynamics. 
%Besides computing the permeability value, we show that the permeability is deterministic. 
%This is a useful findings as it allows to compute the upscaled permeability
%n a deterministic manner and avoid stochastic macroscale PDEs. 

%Even though our application is specific for Brinkman' equations, 
%our mathematical concepts can be used for many important
%applications where the media properties at the microscale are affected
%by SDE. An example can be a diffusion equation with heterogeneous coefficients
%that depend on a field described by SDE. 
%In general, we have 
%\begin{equation}
%\begin{split}
%\dfrac{\partial u^\e}{\partial t} (t, x)= L(a({x\over\e}, v^\e), u^\e)\\
%d v^\e(t, x) = -\dfrac{1}{\e} (v^\e(t, x)-u^\e(t, x)) dt+ \sqrt {\dfrac{Q}{\e}} dW(t, x), 
%\end{split}
%\end{equation}
%where $W(t)$ is a standard Brownian motion. For example, 
%$L(a({x\over\e}, v^\e), u^\e)=\nabla\cdot(a({x\over\e}, v^\e)\nabla u^\e)$. 
%The question of interest is to derive macroscale equations which will be 
%investigated in our future works. 

\subsection{Mathematical model}
Following the previous motivation,  we consider a particular system given by:
\label{subs12}
\begin{equation}
\label{system1}
\left\{
\begin{array}{rll}
\dfrac{\partial u_1^\e}{\partial t} (t, x)&= \operatorname{div} \left(A_1\left(\dfrac{x}{\e}\right) \nabla u^\e(t, x) \right) + \alpha \left(\dfrac{x}{\e}, v_1^\e(t, x), v_2^\e(t, x)\right) \left(u_2^\e(t, x)-u_1^\e(t, x) \right)+ f_1(t, x), \\
\\
\dfrac{\partial u_2^\e}{\partial t} (t, x)&= \operatorname{div} \left(A_2\left(\dfrac{x}{\e}\right) \nabla u^\e(t, x) \right) + \alpha \left(\dfrac{x}{\e}, v_1^\e(t, x), v_2^\e(t, x)\right)  \left(u_1^\e(t, x)-u_2^\e(t, x) \right) + f_2(t, x), \\
d v_1^\e(t, x) &= -\dfrac{1}{\e} (v_1^\e(t, x)-\beta_{11} u_1^\e(t, x) -\beta_{12} u_2^\e(t, x))dt + \sqrt {\dfrac{Q_1}{\e}} dW_1(t, x), \\
d v_2^\e(t, x) &= -\dfrac{1}{\e} (v_2^\e(t, x)-\beta_{21} u_1^\e(t, x) -\beta_{22} u_2^\e(t, x))dt + \sqrt {\dfrac{Q_2}{\e}} dW_2(t, x)\\ 
u_1^\e(t, x)|_{\partial D} &=0, \\
u_2^\e(t, x) |_{\partial D}&=0, \\
u_1^\e(0, x) & = u_{01}^\e(x), \\
u_2^\e(0, x) & = u_{02}^\e(x), \\
v_1^\e(0, x) & = v_{01}^\e(x), \\
v_2^\e(0, x) & = v_{02}^\e(x), 
\end{array}
\right. 
\end{equation}
where $x\in D$ a bounded domain of $\mathbb{R}^3$ with a smooth boundary $\partial{D}$,  and $t\in [0,T]$. $u_1^\e,\ u_2^\e$ and $v_1^\e,\ v_2^\e$ are respectively the components of the velocity of the fluid and the velocity of the particles.   $W_1(t),\ W_2(t)$ are two standard $L^2(D)$-valued independent Brownian motions defined on a complete probability basis $(\Omega, \mathcal{F}, \mathcal{F}_t, {\mathbb P})$ with expectation $\E$, and  $Q_1,\ Q_2$ are bounded linear operators on $L^2(D)$ of trace class. $u^\e_{01},\ u^\e_{02}$ and  $v^\e_{01},\ v^\e_{02}$ are the initial conditions, $\beta_{ij}, \ 1\leq i,j \leq 2$ are constants, and $f_1,\ f_2$ are the external forces. 
 
Our main goal in this paper is to study the asymptotic behavior of the solution of the system \eqref{system1} when $\e\rightarrow 0$. Notice that $u_1^\e, \ u_2^\e$, the slow components, are random through the function $\alpha$ that depends on the stochastic processes $v_1^\e,\ v_2^\e$, the fast ones, solutions of the stochastic differential equations. Moreover, the function $\alpha$ as well as the matrices $A_1=\left(a_{1ij}\right)_{1\leq i, j \leq 3}, \ A_2=\left(a_{2ij}\right)_{1\leq i, j \leq 3}$ are multiscale. We will prove that $u_1^\e,\ u_2^\e$ converge to averaged velocities $\o{u}_1,\ \o{u}_2$ solutions of the averaged system \eqref{eqou1} defined in subsection \ref{sec6}, where the averaged operators $\o{A}_1, \o{A}_2$ and $\o\alpha$ are given by \eqref{homopi} and \eqref{defoalpha}. The averages are taken with respect to the periodic variable $y$ and the invariant measure associated to the process $(v_1^\e,\ v_2^\e)$ for frozen $(u_1^\e, \ u_2^\e)$. 

The techniques used to pass to the limit in the mathematical model \eqref{system1} are a generalization of the techniques used in our recent paper \cite{BEM_19}, where a simpler reaction diffusion model was considered. Let us mention that, like in \cite{BEM_19} and 
\cite{BEM_15} the random coefficient $\alpha$ does depend on the spatial variable $x/\e$ and the process $(v_1^\e, v_2^\e )$ which is ergodic for frozen $(u_1^\e, u_2^\e )$. Hence, the passage to the limit involves a combination of  two kind of convergences: averaging (in time) and homogenization (in space).  There is an extensive literature on averaging principles for stochastic systems \cite{FW1, SV, Ver, CF2008, Cerrai2009} and the references therein. For the basic results on homogenization of periodic and random equations, we refer to \cite{Allaire, BLP}.  Let us refer to an interesting paper \cite{PP} where the authors  studied a particular model of random homogenization, where the coefficients depend upon $x/\e$ and a stationary diffusion process. They effectively used averaging and homogenization techniques like in our current paper. However, the big difference is that our stochastic process $(v_1^\e, v_2^\e )$ is solution of an SDE coupled with the slow motion equation. And a consequence, the ergodic properties of the stochastic process have to be understood when the solution $(u_1^\e, u_2^\e )$ is frozen. This added difficulty makes our convergence process quite different from the convergence process in \cite{PP}.  

%For $\e>0$ fixed, the well posedness of system \eqref{system1} does not follow from classical results and has to be studied accordingly. %In this paper, we assume that $\alpha$ is Lipschitz. In particular, uniqueness of solutions is proved by using successive estimates in %order to get to apply the Gronwall's Lemma, see Section 3. 1. for more details. 

The existence of weak solution for \eqref{system1} is proved in Theorem \ref{thexun}  by using a Galerkin approximation $(u_{1n}^\e,\ u_{2n}^\e,\ v_{1n}^\e, \ v_{2n}^\e)$
that is a solution of a well posed system. We then pass to the limit on $n$ after showing some uniform estimates in $n$. These estimates are also uniform in $\e$. 
By using our assumption on $\alpha$ and the special form of our system, we are able to prove the uniqueness of the weak solution $(u_1^\e, u_2^\e, v_1^\e, v_2^\e)$. We prove that our weak solution is also strong, and get better uniform estimates in $\e$ for the solution $u^\e$ in the Sobolev space $W^{1, 2}(0, T; L^2(D))$. 

We study then the asymptotic behavior of the fast motion variable $(v_1^\e, \ v_2^\e)$ for frozen slow motion variables $(u_1^\e,\ u_2^\e)$. Indeed, we consider the SDEs \eqref{v} for given $\xi=(\xi_1,\xi_2)$. It has a mild solution which is also a strong solution. Its transition semigroup $P_t^{\xi}=P_t^{(\xi_1,\xi_2)}$ is well defined and has a unique invariant measure 
$\mu^\xi=\mu^{(\xi_1,\xi_2)}$ which is ergodic and strongly mixing. 
We define the operators $ \alpha^\e$  and $\o{\alpha^\e}$ in section 3 and  $\o{\alpha^\e}$  
refers to the average of $\alpha^\e$ wrt to the invariant measure 
$\mu^{\xi}$. The main difficulty in showing the convergence stands in passing to the limit for $\phi \in H_0^1(D)$ on the term
\begin{equation}\label{term}
\begin{split}
\int_D \left( \alpha^\e(v_1^\e, v_2^\e) (u_1^\e-u_2^\e) - \o{\alpha}(\beta_{11}\o{u}_1 + \beta_{12}\o{u}_2, \beta_{21}\o{u}_1 + \beta_{22}\o{u}_2)( \o{u}_1-\o{u}_2) \right)\phi dx.
\end{split}
\end{equation}

This is done by using a Khasminskii type argument, following an idea already introduced in \cite{Cerrai2009}. The key lemma \ref{key}, introduced previously in \cite{BEM_19} is crucial for the passage to the limit in \eqref{term}. This lemma is a refined version of previous ergodic results used in \cite{Cerrai2009}. 
By using the uniform estimates obtained in Section 3, a tightness argument 
and some known results for periodic functions, see \cite{Allaire} (lemma 1. 3) the passage to the limit is performed in distribution. We obtain a convergence in probability by using the fact that the limit $\o{u}$ is deterministic. 

The paper is organized as follows, Section 2 is dedicated to the introduction of the functional setting and assumptions. The main results are given in  Section 3. In particular the existence and uniqueness of weak solutions of  system \eqref{system1}, the regularity of solutions and all the uniform estimates wrt $\e$, the cell problems are introduced, the asymptotic behavior of the fast motion is analized and the main result to the averaged system is stated.  All the proofs for the results of Section 3 are postponed to Section 4, including the proof of the main result and the passage to the limit to the averaged equation.

\section{Preliminaries and Assumptions}
\label{sec2}
We make the following notations for spaces that will be used throughout the paper. 
For any two Hilbert spaces $X$ and $Y$, with norms denoted by $\|\cdot\|_X$ and $\|\cdot\|_Y$, $C(X, Y)$ denotes the space of continuous functions, and $C_b(X, Y)$ the Banach space of bounded and continuous functions
$\phi : X\to Y$ endowed with the supremum norm:
$$\|\phi\|_{C_b(X, Y)} = \sup_{x\in X} \|\phi (x)\|_{Y}. $$

For any $\phi\in C^u(X, Y)$, the subspace of uniformly continuous functions defined on $X$ with values in $Y$, we denote by $[\phi]_{C^u(X, Y)} : (0, \infty) \to \mathbb{R}$, the modulus of uniform continuity of $\phi$:
$$[\phi]_{C^u(X, Y)} (r) = \sup_{0 < \|x- y\|_X \leq r} \|\phi (x) - \phi(y)\|_Y, $$
with
$$\lim_{r\to 0} [\phi]_{C_b^u(X, Y)} (r) =0. $$
$Lip(X, Y)$ denotes the space of Lipschitz functions defined on $X$ with values in $Y$, for $\phi \in Lip(X, Y)$ we denote by $[\phi]_{Lip(X, Y)}$ the Lipschitz constant of $\phi$:
$$[\phi]_{Lip(X, Y)} = \sup_{x\neq y} \frac {\|\phi (x) - \phi(y)\|_Y}{\| x-y\|_X}. $$
We notice that for any $\phi \in Lip(X, Y)$ we have:
\begin{equation}\label{lip}
\begin{split}
\|\phi (x)\|_Y \leq \|\phi (x) - \phi(0)\|_Y + \|\phi(0)\|_ Y&\leq [\phi]_{Lip(X, Y)} \|x\|_X + \|\phi(0)\|_Y \\
&\leq ([\phi]_{Lip(X, Y)}+\|\phi(0)\|_Y) (1+\|x\|_X), 
\end{split}
\end{equation}
so the space will be naturally equipped with the norm
\begin{equation}\label{lipnorm}
\|\phi\|_{Lip(X, Y)} = \|\phi(0)\|_Y + [\phi]_{Lip(X, Y)}. 
\end{equation}
To simplify the notations, when there is no confusion we omit the use of subscripts from the notations, and we simply write $\|x\|$, $\|\phi\|$, $[\phi](r)$, $[\phi]$. Also if $Y=\mathbb{R}$ we omit it from the notations, and the spaces are denoted by $C(X)$, $C_b(X)$, $C^u(X)$, and $Lip(X)$. 

For $Y=[0, 1]^3$ the space $C_\#(Y)$ denotes the space of continuous functions on $Y$ that are $Y$-periodic and the space $L^2_\#(Y)$ denotes the closure of $C_\#(Y)$ in $L^2(Y)$. 

We now give the assumptions for the system \eqref{system1} . 

The function $\alpha:Y\times\mathbb{R}^2\to\mathbb{R}$ satisfies the following conditions:

i) For any $\eta_1,\eta_2\in\mathbb{R}$ the function $\alpha(\cdot, \eta_1, \eta_2)$ is measurable and periodic in $y\in Y$. 

ii) For almost every $y\in Y$, the function $\alpha(y, \cdot, \cdot)$ is bounded and Lipschitz, uniformly with respect to $y$. 

We notice that the function $\o{\alpha}:\mathbb{R}^2\to\mathbb{R}$, $\displaystyle\o{\alpha}(\eta_1,\eta_2) = \int_Y \alpha (y, \eta_1,\eta_2) dy$ is Lipschitz and bounded. 

The matrices $A_1=\left(a_{1ij}\right)_{1\leq i, j \leq 3}, \ A_2=\left(a_{2ij}\right)_{1\leq i, j \leq 3} \in L^\infty(Y; \mathbb{R}^{3\times 3})$ are strictly positive and bounded uniformly in $y\in Y$, i. e. there exist $0< m < M$ such that
\begin{equation}\label{A}
m \xi^2 \leq A_1(y)\xi \xi \leq M\xi^2, \  m \xi^2 \leq A_2(y)\xi \xi \leq M\xi^2, 
\end{equation}
for almost every $y\in Y$ and $\xi \in \mathbb{R}^3$. 

We also assume that the external forces $f_1, \ f_2 \in L^2(0, T;L^2(D))$ and the initial conditions $u^\e_{01}. u^\e_{02}, v^\e_{01}, v^\e_{02} \in L^2(D)$. 

%\section{Study of the system \eqref{system1}}

\section{Main results}
\label{sec3}

In this section we state the main results of the paper while their proofs will be postponed to the the following sections.

\subsection{Well-posedness of the system \eqref{system1}}

This subsection will be devoted to stating the existence and uniqueness of the solution of the system \eqref{system1} as well as some uniform estimates. 

 For any $\e>0$ we denote by $A_1^\e, A_2^\e$ the matrices 
\begin{equation}\label{Ae}
A_1^\e, A_2^\e: \mathbb{R}^3 \to \mathbb{R}^{3\times 3}, \ \ A_1^\e(x) = A_1\left( \dfrac{x}{\e} \right), \ A_2^\e(x) = A_2\left( \dfrac{x}{\e}\right)
\end{equation}
and by $\alpha^\e$ the operator, 
\begin{equation}\label{alpha}
\alpha^\e: L^2(D)\times L^2(D) \to L^\infty(D), \ \ \alpha^\e(\eta_1, \eta_2) (x) = \alpha\left(\dfrac{x}{\e}, \eta_1(x), \eta_2(x)\right). 
\end{equation}
The operator $\alpha^\e$ is well defined. Indeed, given that $\alpha$ is bounded, we only need to show the measurability in $x$ of $\alpha^\e(\eta_1,\eta_2)$ for any $\eta_1, \eta_2  \in L^2(D)$. For such a function, we consider two sequences $\eta_{1n}, \eta_{2n} \in C_0(D)$ convergent to $\eta_1, \eta_2$ pointwise in $D$. The function $(y, x) \to \alpha (y, \eta_{1n}(x), \eta_{2n}(x)$ is a Carath\'{e}odory function, measurable in $y$ and continuous in $x$, so $x\to \alpha\left(\dfrac{x}{\e}, \eta_{1n}(x), \eta_{2n}(x) \right)$ is measurable, and by the Lipschitz condition of $\alpha$ is pointwise convergent to $\alpha^\e(\eta_1,\eta_2)$, which shows that $\alpha^\e(\eta_1,\eta_2)$ is measurable. We have the following existence and uniqueness result:

\begin{theorem}\label{thexun}
Assume that $u^\e_{01},u^\e_{02},v^\e_{01},v^\e_{02} \in L^2(D)$ for every $\e>0$, then for each $T>0$,  there exists a unique $\mathcal{F}_t$ - measurable solution of the system \eqref{system1}, $u_1^\e, u_2^\e \in L^\infty(\Omega;C([0, T];L^2(D))\cap L^2(0, T;H_0^1(D)))$ and $v_1^\e, v_2^\e \in L^2(\Omega;C([0, T];L^2(D)))$ in the following sense: $\mathbb{P}$ a. s. 
\begin{equation}
\label{weaksole1}
\begin{split}
&\int_D u_1^\e(t) \phi dx - \int_D u^\e_{01} \phi dx + \int_0^t \int_D A_1^\e\nabla u_1^\e (s) \nabla \phi dx ds =\\
&\int_0^t \int_D \alpha^\e(v_1^\e(s), v_2^\e(s)) (u_2^\e(s)-u_1^\e(s)) \phi dx ds + \int_0^t \int_D f_1(s) \phi dx ds, 
\end{split}
\end{equation}
\begin{equation}
\label{weaksole2}
\begin{split}
&\int_D u_2^\e(t) \phi dx - \int_D u^\e_{02} \phi dx + \int_0^t \int_D A_2^\e\nabla u_2^\e (s) \nabla \phi dx ds =\\
&\int_0^t \int_D \alpha^\e(v_1^\e(s), v_2^\e(s)) (u_1^\e(s)-u_2^\e(s)) \phi dx ds + \int_0^t \int_D f_2(s) \phi dx ds, 
\end{split}
\end{equation}
for every $t\in[0, T]$ and every $\phi \in H_0^1(D)$, and
\begin{equation}
\label{mildsole}
v_1^\e(t) = v^\e_{01} e^{-t/\e} +\frac{1}{\e} \int_0^t (\beta_{11}u_1^\e(s)+ \beta_{12}u_2^\e(s))e^{-(t-s)/\e} ds +\frac{\sqrt{Q_1}}{\sqrt{\e}} \int_0^t e^{-(t-s)/\e} dW_1(s),
\end{equation}
\begin{equation}
\label{mildsole}
v_2^\e(t) = v^\e_{02} e^{-t/\e} +\frac{1}{\e} \int_0^t (\beta_{21}u_1^\e(s)+ \beta_{22}u_2^\e(s))e^{-(t-s)/\e} ds +\frac{\sqrt{Q_2}}{\sqrt{\e}} \int_0^t e^{-(t-s)/\e} dW_2(s). 
\end{equation}

Moreover, if the initial conditions $u^\e_{01}, u^\e_{02}$ are uniformly bounded in $L^2(D)$, then the solutions $u_1^\e, u_2^\e$ satisfy the estimates:
\begin{equation}
\label{est1}
\sup_{\e > 0} \| u_i^\e \|_{L^\infty (\Omega;L^2(0, T;H_0^1(D)))} \leq C_T, 
\end{equation}
\begin{equation}
\label{est2}
\sup_{\e > 0} \| u_i^\e \|_{L^\infty (\Omega;C([0, T];L^2(D)))} \leq C_T, 
\end{equation}
and
\begin{equation}
\label{est3}
\sup_{\e > 0} \left\| \dfrac{\partial u_i^\e}{\partial t} \right\|_{L^\infty (\Omega;L^2(0, T;(H^{-1}(D))))} \leq C_T,
\end{equation}
for $i\in \{1,2\}$. 
Also, if the initial conditions $v^\e_{01}, v^\e_{02}$ are uniformly bounded in $L^2(D)$ we also have the estimates for $v_1^\e, v_2^\e$:
\begin{equation}
\label{est4}
\sup_{\e > 0} \E \sup_{t\in[0, T]} \| v_i^\e(t)\|^2_{L^2 (D)} \leq C_T, 
\end{equation}
for $i\in \{1,2\}$. 
\end{theorem}

\begin{theorem}
\label{threg}
Assume that the initial conditions $u_{01}^\e,\ u_{02}^\e$ are uniformly bounded in $H_0^1(D)$. Then the solutions $u_1^\e, \ u_2^\e \in  L^\infty (\Omega; C([0, T];H^1_0(D)))$ and satisfy the following uniform estimates:
%\begin{equation}
%\label{est1'}
%\sup_{\e > 0} \| u^\e \|_{L^\infty (\Omega;L^2(0, T;H^2(D)))} \leq C_T, 
%\end{equation}
\begin{equation}
\label{est2'}
\sup_{\e > 0} \| u_i^\e \|_{L^\infty (\Omega;C([0, T];H_0^1(D)))} \leq C_T, 
\end{equation}
and
\begin{equation}
\label{est3'}
\sup_{\e > 0} \left\|\dfrac{ \partial u_i^\e}{\partial t} \right\|_{L^\infty (\Omega;L^2(0, T; L^2(D)))} \leq C_T, 
\end{equation}
for $1\leq i,j \leq 2$.
\end{theorem}
\subsection{The cell problems}
\label{sec4}
%Calculation:
%\begin{equation}
%u^\e(x, y, t, s) = u^\e\left(x, \dfrac{x}{\e}, t, \dfrac{t}{\e}\right) = u_0+\e u_1+e^2 u_2
%\end{equation}
%\begin{equation}
%\dfrac{\partial u^\e}{\partial t} = e^{-1} \dfrac{\partial u_0}{\partial s} + \left(\dfrac{\partial u_0} {\partial t} + \dfrac{\partial u_1} {\partial s}\right)
%\end{equation}
%

In this subsection we introduce $\chi_1, \chi_2 : Y \to \mathbb{R}^3$ the solutions of the cell problems that correspond to the system \eqref{system1}:
\begin{equation}
\label{cellpri}
\left\{
\begin{array}{rll}
\operatorname{div} \left(A_i(y) \left( I + \nabla \chi_i(y)\right)\right) &= 0 &\mbox{ in } Y, \\
\chi_i & - Y periodic, \\
\end{array}
\right. 
\end{equation}
for $i\in{1,2}$, and the solutions of the adjoint equations $\chi_1^*, \chi_2^*,$:
\begin{equation}
\label{cellpr*i}
\left\{
\begin{array}{rll}
\operatorname{div} \left(A_i^*(y) \left( I + \nabla \chi_i^*(y)\right)\right) &= 0 &\mbox{ in } Y, \\
\chi_1^* & - Y periodic, \\
\end{array}
\right. 
\end{equation}
where $A_1^*, A_1^*$ are the adjoints of $A_1, A_2$, $A_1^*=((a_1)^*_{ij})_{1\leq i, j \leq 3}$, $(a_1)^*_{ij} = (a_1)_{ji}$ and $A_2^*=((a_2)^*_{ij})_{1\leq i, j \leq 3}$, $(a_2)^*_{ij} = (a_2)_{ji}$ for $1 \leq i, j \leq 3$. It follows that $\chi_1^{\e}(y)=\chi_1\left(\dfrac{y}{\e}\right), \chi_2^{\e}(y)=\chi_2\left(\dfrac{y}{\e}\right)$ are the solutions for the equations:
\begin{equation}
\label{cellpre1}
\left\{
\begin{array}{rll}
\operatorname{div} \left(A_1^\e(y) \left( I + \e\nabla \chi_1^\e(y)\right)\right) &= 0 &\mbox{ in } \e Y, \\
\chi_1^\e& - \e Y periodic, \\
\end{array}
\right. 
\end{equation}

\begin{equation}
\label{cellpre2}
\left\{
\begin{array}{rll}
\operatorname{div} \left(A_2^\e(y) \left( I + \e\nabla \chi_2^\e(y)\right)\right) &= 0 &\mbox{ in } \e Y.\\
\chi_2^\e& - \e Y periodic, \\
\end{array}
\right. 
\end{equation}

We define now the homogenized operator $\o{A}_i$, for $i\in{1,2}$ as
\begin{equation}
\label{homopi}
\o{A}_i=\int_Y A_i(y) \left( I + \nabla \chi_i(y)\right) dy.
\end{equation}

\subsection{The fast motion equation}
\label{sec5}
In this subsection, we present some facts for the invariant measure associated
with (\ref{v}). 
We consider the following problem for fixed $\xi=(\xi_1, \xi_2) \in L^2(D)^2$,  an $L^2(D)^2$-valued Brownian motion $W$ on a probability space $(\Omega, \mathcal{F}, \mathcal{F}_t, {\mathbb P})$ and a bounded linear operator $Q$ on  $L^2(D)^2$ with trace class:

\begin{equation}
\label{v}
\left\{
\begin{array}{ll}
dv^\xi &= - (v^\xi-\xi)dt + \sqrt{Q} dW, \\
v^\xi(0) &= \eta = (\eta_1, \eta_2). 
\end{array}
\right. 
\end{equation}
This equation admits a unique mild solution $v^\xi(t)\in L^2(\Omega; C([0, T];L^2(D)^2))$ given by:
\begin{equation}
\label{vxieta}
v^\xi(t) = \eta e^{-t} +\xi(1-e^{-t}) + \int_0^t e^{-(t-s)}\sqrt{Q} dW. 
\end{equation}
When needed to specify the dependence with respect to the initial condition the solution will be denoted by $v^{\xi, \eta}(t)$. The following estimate can be 
derived for $v^{\xi, \eta}(t)$.

\begin{equation}\label{l}
\E \| v^{\xi, \eta}(t) \|^2_{L^2(D)^2} \leq 2\left(\|\eta\|^2_{L^2(D)^2} e^{-2t} + \|\xi\|^2_{L^2(D)^2} + TrQ \right). 
\end{equation}

%\subsection{The asymptotic behavior of the fast motion equation}
%\label{subs21}
We define the transition semigroups $P_t^\xi$ associated to the equation \eqref{v}:
\begin{equation}
P_t^\xi \Psi (\eta) = \E \Psi(v^{\xi, \eta}(t)), 
\end{equation}
for every $\Psi \in B_b(L^2(D)^2)$, the space of real valued Borel functions defined on $L^2(D)^2$, and every $\eta \in L^2(D)^2$. 
It is easy to verify that $P_t^\xi$ is a Feller semigroup because $\mathbb{P}$ a. s.
\begin{equation}\label{feller}
\|v^{\xi, \eta_1} - v^{\xi, \eta_2}\|^2_{L^2(D)^2} \leq e^{-2t} \|\eta_1-\eta_2 \|^2_{L^2(D)^2}. 
\end{equation}
We also denote by $\mu^\xi$ the associated invariant measure on $L^2(D)^2$. We recall that it is invariant for the semigroup $P_t^\xi$ if
$$\int_{L^2(D)^2} P_t^\xi \Psi (z) d\mu^\xi(z) = \int_{L^2(D)^2} \Psi (z) d\mu^\xi(z), $$
for every $\Psi \in B_b(L^2(D)^2)$. 
It is obvious that $v^\xi$ is a stationary gaussian process. The equation \eqref{v} admits a unique ergodic invariant measure $\mu^\xi$ that is strongly mixing and gaussian with mean $\xi$ and covariance operator $Q$. All these results can be found in \cite{DPZ2} or \cite{Cerrai}. 

As a consequence of \eqref{feller} we also have:
\begin{equation}\label{inv}
\left |P_t^\xi \Phi (\eta) - \int_{L^2(D)^2} \Phi (z) d\mu^\xi(z)\right |\leq c[\Phi] e^{-t}(1+\|\eta\|_{L^2(D)^2} +\|\xi \|_{L^2(D)^2}), 
\end{equation}
for any Lipschitz function $\Phi$ defined on $L^2(D)^2$, where $[\Phi]$ is the Lipschitz constant of $\Phi$. 

%This can be shown as it follows:
%\begin{equation}
%\begin{split}
%P_t^\xi \Phi (\eta) - \int_{L^2(D)^2} \Phi (z) d\mu^\xi(z) &= \int_{L^2(D)^2} \left(P_t^\xi \Phi (\eta) -P_t^\xi \Phi (z)\right) d\mu^\xi(z)\\
%&= \int_{L^2(D)^2} \left(\E \Phi (v^{\xi, \eta}(t)) - \E \Phi (v^{\xi, z}(t))\right) d\mu^\xi(z) \\
%& \leq \int_{L^2(D)^2} [\Phi] \E \left\|v^{\xi, \eta}(t) - v^{\xi, z}(t)\right\|_{L^2(D)^2} d\mu^\xi(z)\\
%&\leq \int_{L^2(D)^2} [\Phi] e^{-t}\E \left\|\eta-z\right\|_{L^2(D)^2} d\mu^\xi(z)\\
%&\leq [\Phi] e^{-t} \left(\|\eta\|_{L^2(D)^2} + \int_{L^2(D)^2} \left\|z\right\|_{L^2(D)^2} d\mu^\xi(z)\right). 
%\end{split}
%\end{equation}
%Now \eqref{inv} follows as a result of the following lemma:
%\begin{lemma}
%\label{mod}
%\begin{equation}
 %\int_{L^2(D)^2} \left\| z \right\|_{L^2(D)^2} d\mu^\xi(z)\leq c\left(1+\left\|\xi\right\|_{L^2(D)^2}\right). 
 %\end{equation}
 %\begin{proof}
%\begin{equation}L^2(D)^2
%\begin{split}
%\int_{L^2(D)^2} \|z\|_{L^2(D)^2}d\mu^\xi(z) &= \int_{L^2(D)^2} P_t^\xi \|z\|_{L^2(D)^2}d\mu^\xi(z)\\
%& = \int_{L^2(D)^2} \E \|v^{\xi, z}(t)\|_{L^2(D)^2}d\mu^\xi(z)\\
%&\leq \int_{L^2(D)^2} c(1 + \|\xi\|_{L^2(D)^2}+ e^{-t}\|z\|_{L^2(D)^2})d\mu^\xi(z). 
%\end{split}
%\end{equation}
%We fix now $t>0$ and get the result. 
%\end{proof}
%\end{lemma}

As described in the introduction, in order to pass to the limit on some terms of equation $\eqref{weaksole1}$ and  $\eqref{weaksole2}$, we will need to use the ergodic properties of the fast motion. However, the estimate \eqref{inv} is not enough and we will need to use a more refined ergodic result.  The remark and the lemma below have been introduced in our previous paper, see \cite{BEM_19}  in order to analyze a similar model. Indeed, the use of Lemma \ref{key} is essential to the analysis of our mathematical model \eqref{system1}.
 
\begin{remark}
For $\xi, \eta \in L^2(\Omega, \mathcal{F}_{t_0}, L^2(D)^2)$, let $v^{\xi, \eta}$ be the solution of the following system, the equivalent of the system \eqref{v} but with random initial conditions $\eta$ and random parameter $\xi$:
\begin{equation}
\label{vr}
\left\{
\begin{array}{ll}
dv^{\xi, \eta} &= - (v^{\xi, \eta}-\xi)dt + \sqrt{Q} dW, \\
v^{\xi, \eta}(t_0) &= \eta. 
\end{array}
\right. 
\end{equation}
The mild solution for \eqref{vr} $v^{\xi, \eta}(t)\in L^2(\Omega; C([t_0, T];L^2(D)^2))$ exists and is given by:
\begin{equation}
\label{vxietar}
v^{\xi, \eta}(t) = \eta e^{-(t-t_0)} +\xi(1-e^{-(t-t_0)}) + \int_0^{(t-t_0)} e^{-(t-t_0-s)}\sqrt{Q} dW. 
\end{equation}

The estimates provided by \eqref{l} and \eqref{inv} remains\ valid also in the case when $\xi$ and $\eta$ are random. So for any $\xi, \eta \in L^2(\Omega, \mathcal{F}_{t_0}, L^2(D)^2)$ we have:

\begin{equation}\label{l1}
\E \left( \|v^{\xi, \eta}(t)\| ^2_{L^2(D)^2}|\mathcal{F}_{t_0} \right)\leq 2\left(\|\eta\|^2_{L^2(D)^2} e^{-2(t-t_0)} + \|\xi\|^2_{L^2(D)^2}+ Tr Q\right), 
\end{equation}
and

\begin{equation}\label{invr}
\E\left(\left |(P_t)^{\xi(\omega)} \Phi (\eta(\omega)) - \int_{L^2(D)^2} \Phi (z) d\mu^{\xi(\omega)}(z)\right |\Big | \mathcal{F}_{t_0}\right) \leq c[\Phi] e^{-(t-t_0)}(1+\|\eta(\omega)\|_{L^2(D)^2} +\|\xi (\omega)\|_{L^2(D)^2}), 
\end{equation}
a. s. $\omega \in \Omega$, for any Lipschitz function $\Phi$ defined on $L^2(D)^2$. 
\end{remark}
The equation \eqref{invr} implies the following Lemma that has been first introduced in \cite{BEM_19} where a detailed proof can be found.

\begin{lemma}
\label{key}
Let $\Phi \in C^u([0, T]; L^\infty(\Omega;Lip(L^2(D)^2)))$ be an $\mathcal{F}_t$ - measurable process on $Lip(L^2(D)^2)$, and let $0\leq t_0 <t_0+\delta \leq T$. For $\xi, \eta \in L^2(\Omega, \mathcal{F}_{t_0}, L^2(D)^2)$, let $v^{\xi, \eta}$ be the solution of the system \eqref{vr}. We have:
\begin{equation}\label{estkey}
\begin{split}
&\E \left(\left| \frac{1}{\delta} \int_{t_0}^{t_0+\delta}\left( \Phi(s, v^{\xi, \eta}(s)) - \int_{L^2(D)^2}\Phi (s, z) d\mu^{\xi}(z)\right)ds\right| \Big | \mathcal{F}_{t_0}\right)\leq\\
 & c\left(1+\|\eta\|_{L^2(D)^2}+\|\xi\|_{L^2(D)^2}\right) \left(\frac{\|\Phi\|}{\sqrt{\delta}} +\sqrt{\|\Phi\|[\Phi] (\delta)} \right), 
\end{split}
\end{equation}
where $[\Phi]$ is the modulus of uniform continuity of $\Phi$ and $c$ is a positive constant. 
\end{lemma}
\begin{proof} See \cite{BEM_19}
\end{proof}

%\section{Passage to the limit}
\subsection{Main result: The averaged system}
\label{sec6}

We introduce the following averaged operators:
\begin{equation}
\label{defoalphae}
\o{\alpha^\e}: L^2(D)^2  \to L^\infty(D), \ \ \o{\alpha^\e}(\xi_1, \xi_2) = \int_{L^2(D)^2}\alpha^\e(\eta_1, \eta_2) d\mu^{\xi} (\eta_1, \eta_2)
\end{equation}

\begin{equation}
\label{defoalpha}
\o{\alpha}: L^2(D)^2  \to L^\infty(D), \ \ \o{\alpha} (\xi_1, \xi_2) = \int_{L^2(D)^2} \left(\int_Y \alpha(y, z_1, z_2)dy\right) d\mu^{\xi} (z_1, z_2)\end{equation}

We remark that $\alpha^\e$ as an operator from $L^2(D)\times L^2(D)$ to $L^2(D)$ is Lipschitz and $L^2(D)$ is separable, so Pettis Theorem implies that $\alpha^\e: L^2(D)\times L^2(D) \to L^2(D)$ is measurable. The boundedness of $\alpha^\e$ implies the integrability with respect to the probability measure $\mu^{\xi} (\eta_1, \eta_2)$ is well defined (see Chapter 5, Sections 4 and 5 from \cite{Yosida} for details). The same considerations hold also for the operators $(z_1, z_2)\in L^2(D) \times L^2(D) \to \o{\alpha} (z_1,z_2)=\displaystyle\int_Y \alpha(y, z_1,z_2)dy \in L^\infty(D)$, so $\o{\alpha}$ is also well defined. 
Our main result is given in the next theorem.
\begin{theorem}
\label{thconv1}
Assume the sequences $u^\e_{01}, u^\e_{02}$ are uniformly bounded in $H_0^1(D))$ and strongly convergent in $L^2(D)$ to some functions $u_{01}, u_{02}$, and $v^\e_{01}, v^\e_{02}$ are uniformly bounded in $L^2(D)$. Then, there exist $\o{u}_1, \o{u}_2 \in L^2(0, T;H_0^1(D))\cap C([0, T]; L^2(D))$ such that $u_1^\e, u_2^\e$ converge in probability to $\o{u}_1, \o{u}_2$ in $w\mbox{-}L^2(0, T;H_0^1(D))\cap C([0, T]; L^2(D))$ and $\{\o{u}_1, \o{u}_2\}$ is the solution of the following deterministic equation:
\begin{equation}
\label{eqou1}
\left\{
\begin{array}{rll}
\dfrac{\partial \o{u}_1}{\partial t} &= \operatorname{div} \left( \o{A}_1 \nabla\o{u}_1\right) + \o{\alpha} (\beta_{11}\o{u}_1+\beta_{12}\o{u}_2,\beta_{21}\o{u}_1+\beta_{22}\o{u}_2) (\o{u}_2-\o{u}_1) + f_1 &\mbox{ in }\ D, \\
\dfrac{\partial \o{u}_2}{\partial t} &= \operatorname{div} \left( \o{A}_2 \nabla\o{u}_2\right) + \o{\alpha} (\beta_{11}\o{u}_1+\beta_{12}\o{u}_2,\beta_{21}\o{u}_1+\beta_{22}\o{u}_2) (\o{u}_1-\o{u}_2) + f_2 &\mbox{ in }\ D, \\
\o{u}_1 &=0 &\mbox{ on }\ \partial D, \\
\o{u}_2 &=0 &\mbox{ on }\ \partial D, \\
\o{u}_1(0) & = u_{01}&\mbox{ in }\ D, \\
\o{u}_2(0) & = u_{02}&\mbox{ in }\ D. 
\end{array}
\right. 
\end{equation}
\end{theorem}

\subsection{Well-possedness for the averaged equation \eqref{eqou1}} \label{subs41}

We state here that the averaged  system \eqref{eqou1} is well posed while its proof will be postponed to the section on proofs. 
 
 \begin{theorem}\label{thexunou}
Assume $f_1, \ f_2 \in L^2(0, T;L^2(D))$ and $\o{\alpha} \in Lip_b(\mathbb{R}^2)$. Then, for any $u_{01},\ u_{02} \in L^2(D)$ the system \eqref{eqou1} admits a unique solution $\o{u}_1,\ \o{u}_2 \in C([0, T];L^2(D))\cap L^2(0, T;H_0^1(D)))$ with $\dfrac{\partial \o{u}_1}{\partial t} , \ \dfrac{\partial \o{u}_2}{\partial t} \in L^2(0, T;H^{-1}(D))$ in the following sense:
\begin{equation}
\label{weaksolou}
\begin{split}
&\int_D \o{u}_1(t) \phi dx - \int_D u_{01} \phi dx + \int_0^t \int_D \o{A}_1\nabla \o{u}_1 (s) \nabla \phi dx ds =\\
&\int_0^t \int_D \o{\alpha}(\beta_{11}\o{u}_1 + \beta_{12}\o{u}_2, \beta_{21}\o{u}_1 + \beta_{22}\o{u}_2) (\o{u}_2-\o{u}_1) \phi dx ds + \int_0^t \int_D f_1(s) \phi dx ds\\
&\int_D \o{u}_2(t) \phi dx - \int_D u_{02} \phi dx + \int_0^t \int_D \o{A}_2\nabla \o{u}_2 (s) \nabla \phi dx ds =\\
&\int_0^t \int_D \o{\alpha}(\beta_{11}\o{u}_1 + \beta_{12}\o{u}_2, \beta_{21}\o{u}_1 + \beta_{22}\o{u}_2) (\o{u}_1-\o{u}_2) \phi dx ds + \int_0^t \int_D f_2(s) \phi dx ds,
\end{split}
\end{equation}
for every $t\in[0, T]$ and every $\phi \in H_0^1(D)$. Moreover, if the initial condition $u_{01}, \ u_{02} \in H_0^1(D)$, then $\o{u}_1, \ \o{u}_2$ have the improved regularity, $\o{u}_1, \ \o{u}_2 \in  L^\infty(0, T;H_0^1(D))$ and $\dfrac{ \partial \o{u}_1}{\partial t} , \ \dfrac{ \partial \o{u}_2}{\partial t}\in L^2(0, T; L^2(D))$. 
\end{theorem}

\section{Proofs}

\subsection{Proof of Theorem \ref{thconv1}}
\label{subs45}
In order to prove the theorem, we first need to prove that

\begin{equation}
\label{dif}
\begin{split}
\lim_{\e \to 0} &\ \E \left | \int_0^T \int_D \left( \alpha^\e(v^\e_1, v^\e_2) (u_2^\e-u_1^\e) - \o{\alpha}(\beta_{11}\o{u}_1+ \beta_{12}\o{u}_2, \beta_{21}\o{u}_1+ \beta_{22}\o{u}_2) (\o{u}_2-\o{u}_1) \right)\phi^\e \psi dx dt \right | = 0, 
\end{split}
\end{equation}
for a particular sequence $\phi^\e \in H^1_0(D)$ and $\psi \in C[0, T]$. 
We rewrite it as a sum:
$$\int_0^T \int_D \left( \alpha^\e(v^\e_1, v^\e_2) (u_2^\e-u_1^\e) - \o{\alpha}(\beta_{11}\o{u}_1 + \beta_{12}\o{u}_2, \beta_{21}\o{u}_1 + \beta_{22}\o{u}_2) (\o{u}_2-\o{u}_1) \right)\phi^\e \psi dx dt = S^\e_1 + S^\e_2 +S^\e_3, $$
where
\begin{equation}\nonumber
S^\e_1 = \int_0^T \int_D \left( \alpha^\e(v^\e_1, v^\e_2) (u_2^\e-u_1^\e) - \o{\alpha^\e}(\beta_{11}u^\e_1 + \beta_{12}u^\e_2, \beta_{21} u^\e_1 + \beta_{22}u^\e_2) (u_2^\e-u_1^\e) \right)\phi^\e \psi dx dt, 
\end{equation}
\begin{equation}\nonumber
S^\e_2 = \int_0^T \int_D \left( \o{\alpha^\e}(\beta_{11}u^\e_1 + \beta_{12}u^\e_2, \beta_{21} u^\e_1 + \beta_{22}u^\e_2) (u_2^\e-u_1^\e)- \o{\alpha^\e}(\beta_{11}\o{u}_1 + \beta_{12}\o{u}_2, \beta_{21}\o{u}_1 + \beta_{22}\o{u}_2) (\o{u}_2-\o{u}_1) \right)\phi^\e \psi dx dt, 
\end{equation}
and
\begin{equation}\nonumber
S^\e_3 = \int_0^T \int_D \left(\o{\alpha^\e}(\beta_{11}\o{u}_1 + \beta_{12}\o{u}_2, \beta_{21}\o{u}_1 + \beta_{22}\o{u}_2) (\o{u}_2-\o{u}_1) - \o{\alpha}(\beta_{11}\o{u}_1 + \beta_{12}\o{u}_2, \beta_{21}\o{u}_1 + \beta_{22}\o{u}_2) (\o{u}_2-\o{u}_1)\right)\phi^\e \psi dxdt. 
\end{equation}
The convergence to $0$ for $S^\e_1$ is performed by proving the more general result \eqref{converg} where the equation satisfied by $u^\e$ is not important. The idea of proving \eqref{converg} is to approximate $u^\e$ and $\phi^\e$ by step functions in time and use Lemma \ref{key} on each piece. Then the convergence of $S^\e_2$ and  $S^\e_3$ to 0 are proved below. 

 The sequence $\wt{u_1^\e}$ given by Skorokhod theorem converges a. s. to $\wt{\o{u}_1}$ weakly in $L^2(0, T;H^1_0(D))$ and strongly in $C([0, T];L^2(D))$ so 
\begin{equation}\label{dif1}
\lim_{\e \to 0} \left | \int_0^T \int_D \left(\wt{u_1^\e} (t)-\wt{\o{u}_1} (t)\right) \phi^\e \psi'(t) dx dt - \int_0^T \int_D \left( A_1^\e \nabla \wt{u_1^\e} - \o{A}_1\nabla\wt{\o{u}_1} \right) \nabla \phi \psi(t) dx dt\right | = 0, \quad a. s. \end{equation}
%pointwise in $\Omega '$. 
The equations \eqref{dif} and \eqref{dif1} imply that $\wt{\o{u}_1}$ satisfies almost surely the variational formulation associated with \eqref{eqou1}, so $\wt{\o{u}_1}$ and $\o{u}_1$ are deterministic and as a consequence the convergence of the sequence $u_1^\e$ to $\o{u}_1$ will be in probability. Similarly we get the convergence for $u_2^\e$ to $\o{u}_2$.

\subsubsection{Convergence of $S^\e_1$}\label{subs42}

\begin{lemma}
\label{converg}
Assume that $u^\e$ is a sequence of $\mathcal{F}_t$ - measurable processes in $L^2(D)^2$, uniformly bounded in $L^\infty(\Omega, W^{1, 2}(0, T;L^{2}(D)^2))$, $\phi^\e$ a sequence of $\mathcal{F}_t$ - measurable processes in $L^2(D)^2$, such that $\phi^\e \in L^\infty(\Omega;C^u([0, T] ;L^2(D)^2))$ uniformly bounded and equiuniform continuous with respect to $\e >0$ and $\omega\in\Omega$. Let the sequence $v^\e$ satisfy the equation
\begin{equation} \label{ve}
\left\{
\begin{array}{rll}
d v^\e(t, x) &= -\dfrac{1}{\e} (v^\e(t, x)-u^\e(t, x))dt + \sqrt {\dfrac{Q}{\e}} dW(t, x) &\mbox{ in }\ [0, T] \times D, \\
v^\e(0, x) & = v_0^\e(x)&\mbox{ in }\ D, 
\end{array}\right. 
\end{equation}
with the sequence $v_0^\e$ uniformly bounded in $L^2(D)^2$. Then we have that:
\begin{equation}
\lim_{\e\to 0} \E \left| \int_0^T \int_D \left(\alpha^\e(v^\e(t)) - \o{\alpha}^\e (u^\e(t))\right) \phi^\e(t) dx dt\right| =0. 
\end{equation}
\end{lemma}
\begin{proof}
Fix $n^\e$ a positive integer and let $\delta^\e =\dfrac{T}{n^\e}$. We define $\wt{u}^\e$ as the piecewise constant function:
\begin{equation}\label{wtphie}
\wt{u}^\e(t)=u^\e(k\delta^\e) \ \mbox{ for } t\in[k\delta^\e, (k+1)\delta^\e). 
\end{equation}
We define also the sequence $\wt{v}{^\e}$ as the solution of:
\begin{equation} \label{wtve}
\left\{
\begin{array}{rll}
d \wt{v}^\e(t, x) &= -\dfrac{1}{\e} (\wt{v}^\e(t, x)-\wt{u}^\e(t, x))dt + \sqrt {\dfrac{Q}{\e}} dW(t, x) &\mbox{ in }\ [0, T] \times D, \\
\wt{v}^\e(0, x) & = v_0^\e(x)&\mbox{ in }\ D. 
\end{array}
\right. 
\end{equation}
A simple calculation shows that the sequence $u^\e$ is H\"{o}lder continuous, uniformly in $\e$ and $\omega$:
\begin{equation}\nonumber
\begin{split}
u^\e(t) - u^\e(s) &=\int_s^t \dfrac{\partial u^\e}{\partial t}(r) dr \Rightarrow\\
\|u^\e(t) - u^\e(s)\|_{L^2(D)^2}&\leq (t-s)^{\frac{1}{2}}\left(\int_0^T \left\| \dfrac{\partial u^\e}{\partial t}(r)\right\|^2_{L^{2}(D)} dr \right)^{\frac{1}{2}}\leq C(t-s)^{\frac{1}{2}}. 
\end{split}
\end{equation}
This implies that:
\begin{equation}\label{difue}
\lim_{\delta^\e \to 0} \| \wt{u}^\e - u^\e \|_{L^\infty(0, T; L^2(D)^2)} =0, 
\end{equation}
uniformly in $\e$ and $\omega$. 
From \eqref{ve} and \eqref{wtve} we get that $\wt{v}^\e(t) - v^\e(t) = \dfrac{1}{\e}\displaystyle\int_0^t e^{\frac{-(t-s)}{\e}} \left(\wt{u}^\e(s) - u^\e(s)\right)ds$, so we also have that
\begin{equation}\label{difve}
\lim_{\delta^\e \to 0} \| \wt{v}^\e - v^\e \|_{L^\infty(0, T; L^2(D)^2)} =0, 
\end{equation}
uniformly in $\e$ and $\omega$. 

Now
\begin{equation}\nonumber
\begin{split}
\int_0^T \int_D \left(\alpha^\e(v^\e(t)) - \o{\alpha}^\e (u^\e(t))\right) \phi^\e(t) dx dt -\int_0^T \int_D \left(\alpha^\e(\wt{v}^\e(t)) - \o{\alpha}^\e (\wt{u}^\e(t))\right) \phi^\e(t) dx dt=\\
\int_0^T \int_D \phi^\e(t)\left(\alpha^\e(v^\e(t))-\alpha^\e(\wt{v}^\e(t)) \right) dxdt + \int_0^T \int_D \phi^\e(t)\left(\o{\alpha}^\e(\wt{u}^\e(t))-\o{\alpha}^\e(u^\e(t)) \right) dxdt, \end{split}
\end{equation}
But:
\begin{equation}\nonumber
\begin{split}
&\int_0^T \int_D \phi^\e(t)\left(\alpha^\e(v^\e(t))-\alpha^\e(\wt{v}^\e(t)) \right) dxdt\leq\\
&\|\phi^\e\|_{L^\infty(\Omega; C([0, T]; L^2(D)^2))}\int_0^T \left(\int_D |\alpha^\e(v^\e(t)) - \alpha^\e (\wt{v}^\e(t)|^{2}dx \right)^{1/2}\leq\\
&\|\phi^\e\|_{L^\infty(\Omega; C([0, T]; L^2(D)^2))} \int_0^T \left(\int_D [\alpha]^2 \left|v^\e(t) - \wt{v}^\e(t)\right|^2 dx \right)^{1/2}\leq\\
&C\sqrt{T} \|\phi^\e\|_{L^\infty(\Omega; C([0, T]; L^2(D)^2))} [\alpha]\| \wt{v}^\e - v^\e \|_{L^\infty(0, T; L^2(D)^2)}, 
\end{split}
\end{equation}
and similarly 
\begin{equation}\nonumber
\begin{split}
&\int_0^T \int_D \phi^\e(t)\left(\o{\alpha}^\e(\wt{u}^\e(t))-\o{\alpha}^\e(u^\e(t)) \right) dxdt\leq\\
&C \sqrt{T}\|\phi^\e\|_{C([0, T]; L^\infty(\Omega;L^2(D)^2))} [\alpha] \| \wt{u}^\e - u^\e \|^{1/2}_{L^\infty(0, T; L^2(D)^2)}, 
\end{split}
\end{equation}
which will imply based on \eqref{difue} and \eqref{difve} that
\begin{equation}\label{difalpha}
\lim_{\delta^\e \to 0}\E\left| \int_0^T \int_D \left(\alpha^\e(v^\e(t)) - \o{\alpha}^\e (u^\e(t))\right) \phi^\e(t) dx dt -\int_0^T \int_D \left(\alpha^\e(\wt{v}^\e(t)) - \o{\alpha}^\e (\wt{u}^\e(t))\right) \phi^\e(t) dx dt\right| =0, 
\end{equation}
uniformly in $\e$. 

Let us study now the term $\displaystyle \int_0^T \int_D \left(\alpha^\e(\wt{v}^\e(t)) - \o{\alpha}^\e (\wt{u}^\e(t))\right) \phi^\e(t) dx dt$. 
\begin{equation}
\begin{split}\label{eq3}
\int_0^T \int_D \left(\alpha^\e(\wt{v}^\e(t)) - \o{\alpha}^\e (\wt{u}^\e(t))\right) \phi^\e(t) dx dt &=\sum_{k=0}^{n^\e-1} \int_{k \delta^\e}^{(k+1)\delta^\e} \int_D\left(\alpha^\e(\wt{v}^\e(t)) - \o{\alpha}^\e (\wt{u}^\e(t))\right) \phi^\e(t) dx dt. \end{split}
\end{equation}

The process defined by 
\begin{equation}\label{deffe}
F^\e(s, \eta)= \displaystyle\int_D \alpha^\e(\eta) \phi^\e \left(\e s\right)dx
\end{equation}
belongs to $C^u([0, T / \e ] ;Lip(L^2(D)^2))$, 
with 
$$| F^\e (s, 0) | \leq |\alpha| \|\phi^\e\|_{C([0, T]; L^2(D)^2)}, $$
$$ [F^\e(s, \cdot)] \leq [\alpha] \left\| \phi^\e\right\|_{C([0, T];L^2(D)^2)}, $$ 
so
$$ \| F^\e (s)\|_{Lip(L^2(D)^2)} \leq (|\alpha|+[\alpha]) \|\phi^\e\|_{C([0, T]; L^2(D)^2)}$$
and 
$$[F^\e] (r) \leq(|\alpha|+[\alpha]) [\phi^\e]_{C^u([0, T]; L^2(D)^2)}(\e r), $$ so we can apply Lemma \ref{key} on the interval $[k \delta_\e / \e, (k+1) \delta_\e / \e]$ for $\xi = u^\e(k\delta^\e)$ and $\eta=\wt{v}^\e(k \delta^\e)$ to the sequence $ F^\e$:

\begin{equation}\label{eq1}
\begin{split}
&\E \left(\left| \frac{\e}{\delta^\e} \int_{k\delta^\e / \e}^{(k+1)\delta^\e / \e} F^\e(s, v^{u^\e(k\delta^\e), \wt{v}^\e(k \delta^\e)}(s)) ds- \int_{L^2(D)^2} F^\e (s, z) d\mu^{u^\e(k\delta^\e)}(z)\right| \Big | \mathcal{F}_{k\delta^\e}\right)\leq \\
&c \left(1+\|\wt{v}^\e(k \delta^\e)\|_{L^2(D)^2}+\|u^\e(k\delta^\e)\|_{L^2(D)^2}\right)\left( \frac{\sqrt{\e}\|F^\e\| }{\sqrt{\delta_\e}} + \sqrt{\|F^\e\| [F^\e]( \delta_\e / \e)} \right)\leq \\
&C \left(1+\|\wt{v}^\e(k \delta^\e)\|_{L^2(D)^2}+\|u^\e(k\delta^\e)\|_{L^2(D)^2}\right)\left( \frac{\sqrt{\e}\|\phi^\e\|}{\sqrt{\delta_\e}} + \sqrt{\|\phi^\e\|\left[ \phi^\e \right](\delta_\e )} \right). \\
\end{split}
\end{equation}

But by a change of variables $\wt{v}^\e\left(\e t\right)$ is a solution for the equation \eqref{vr} on the interval $[k\delta^\e / \e, (k+1)\delta^\e / \e]$ with $\xi = u^\e(k \delta^\e)$ and $\eta = \wt{v}^\e(k \delta^\e)$, so 
$$ v^{u^\e(k\delta^\e), \wt{v}^\e(k \delta^\e)}(s) = \wt{v}^\e\left(\e s\right). $$
Also using formula \eqref{l1}:
\begin{equation}\nonumber
\begin{split}
\E\left(\| \wt{v}^\e((k+1)\delta^\e)\|^2_{L^2(D)^2}|\mathcal{F}_{k\delta^\e}\right) \leq &c\left(\| \wt{v}^\e(k\delta^\e)\|^2_{L^2(D)^2} e^{-2\delta^\e / \e} + \| u^\e(k\delta^\e)\|^2_{L^2(D)^2}+1\right)\Rightarrow\\
\| \wt{v}^\e((k+1)\delta^\e)\|^2_{L^2(\Omega, L^2(D)^2)}\leq & c \left(\| \wt{v}^\e(k\delta^\e)\|^2_{L^2(\Omega, L^2(D)^2)} e^{-2\delta^\e/\e} + \| u^\e\|^2_{L^2(\Omega, C([0, T];L^2(D)^2))}+1\right), 
\end{split}
\end{equation}
and we obtain by induction that:
\begin{equation}\nonumber
\| \wt{v}^\e(k\delta^\e)\|^2_{L^2(\Omega, L^2(D)^2)} \leq c^ke^{-2k\delta^\e/\e} \| \wt{v}^\e(0\|^2_{L^2(\Omega, L^2(D)^2)} +\left(\sum_{i=1}^k c^k e^{-2k\delta^\e/\e}\right) \left(\| u^\e\|^2_{L^2(\Omega, C([0, T];L^2(D)^2))}+1\right), 
\end{equation}
so for $\e/\delta^\e$ small enough we get the estimate:
\begin{equation}\label{estweve}
\| \wt{v}^\e(k\delta^\e)\|^2_{L^2(\Omega, L^2(D)^2)} \leq C \left(\| u^\e\|^2_{L^2(\Omega, C([0, T];L^2(D)^2))}+1\right)\ ,\ \forall k >0. 
\end{equation}
The equation \eqref{eq1} now becomes:
\begin{equation}\label{eq2}
\begin{split}
&\E \left| \frac{\e}{\delta^\e} \int_{k\delta^\e / \e}^{(k+1)\delta^\e / \e} F^\e\left(s, \wt{v}^\e\left(\e s\right)\right) ds- \int_{L^2(D)^2}F^\e \left(s, z\right) d\mu^{u^\e(k\delta^\e)}(z)\right| =\\
&\E \left| \frac{1}{\delta^\e} \int_{k\delta^\e}^{(k+1)\delta^\e} F^\e(\frac{ s}{\e}, \wt{v}^\e\left(s\right)) ds- \int_{L^2(D)^2}F^\e (\frac{ s}{\e}, z) d\mu^{u^\e(k\delta^\e)}(z)\right|\leq \\
&C \left(1+\| u^\e\|^2_{L^2(\Omega, C([0, T];L^2(D)^2))}\right)\left( \frac{\sqrt{\e}\|\phi^\e\|}{\sqrt{\delta_\e}} + \sqrt{\|\phi^\e\|\left[\phi^\e \right](\delta_\e )} \right). 
\end{split}
\end{equation}
If we sum over all $0\leq k \leq n^\e-1$ and go back to the equation \eqref{eq3} we obtain that
\begin{equation}
\begin{split} 
&\E \left| \int_0^T \int_D \left(\alpha^\e(\wt{v}^\e(t)) - \o{\alpha}^\e (\wt{u}^\e(t))\right) \phi^\e(t) dx dt\right| \\
\leq & C \left(1+\|u^\e\|_{C([0, T]; L^\infty(\Omega, L^2(D)^2))}\right)\left( \frac{\sqrt{\e}\|\phi^\e\| }{\sqrt{\delta_\e}} + \sqrt{\|\phi^\e\|\left[ \phi^\e\right](\delta_\e )} \right). 
\end{split} 
\end{equation}
If we choose now $n^\e = T/ \sqrt{\e}$ use the equiuniform continuity of $\phi^\e$ and the convergences given by \eqref{difalpha} we obtain that 
$$\lim_{\e\to 0}\E \left| \int_0^T \int_D \left(\alpha^\e(v^\e(t)) - \o{\alpha}^\e (u^\e(t))\right) \phi^\e(t) dx dt\right| =0, $$
which proves the Lemma. 
\end{proof}
The convergence to $0$ of $S_1^\e$ follows:
\begin{lemma}\label{lemmaconv1}
If $\phi^\e$ is a sequence uniformly bounded in $H_0^1(D)$ and $\psi\in C[0, T]$ then:
\begin{equation}
\label{convre}
\begin{split}
\lim_{\e\to 0} \E\left|  \int_0^T \int_D \left( \alpha^\e(v^\e_1, v^\e_2) (u_2^\e-u_1^\e) - \o{\alpha^\e}(\beta_{11}u^\e_1 + \beta_{12}u^\e_2, \beta_{21} u^\e_1 + \beta_{22}u^\e_2) (u_2^\e-u_1^\e) \right)\phi^\e \psi dx dt \right| =& 0. 
\end{split}
\end{equation}
\begin{proof}
As $u^\e_1, \ u^\e_2$ are uniformly bounded in $L^\infty(\Omega, C([0, T]; H_0^1(D))) \cap L^\infty(\Omega, W^{1, 2}(0, T; L^2(D)))$ and $\Psi\in C[0, T]$, then the sequences $(u^\e_2-u^\e_1) \phi^\e \psi $ is uniformly bounded and equiuniformly continuous in $C([0, T] ; L^\infty(\Omega;L^2(D)))$. We then apply thes Lemma for $v^\e=(v^\e_1,v^\e_2)$ and $u^\e=(\beta_{11}u^\e_1 + \beta_{12}u^\e_2, \beta_{21} u^\e_1 + \beta_{22}u^\e_2)$.
\end{proof}
\end{lemma}
\subsubsection{Convergence of $S^\e_2$}\label{subs43}
\begin{lemma}\label{lemmaconv1''}
Assume $u_1^\e, \ u_2^\e$ are two sequences uniformly bounded in $L^\infty(\Omega, C([0, T], H_0^1(D)))$ that converge in distribution to $\o{u}_2, \ \o{u}_1$ in $C([0, T], L^2(D)))$. Then, for any sequence $\phi^\e$ uniformly bounded in $H_0^1(D)$ and $\psi\in C[0, T]$ we have:
\begin{equation}\label{convre''}
\lim_{\e\to 0}\E \left|\int_0^T \int_D \left( \o{\alpha^\e}(\beta_{11}u^\e_1 + \beta_{12}u^\e_2, \beta_{21} u^\e_1 + \beta_{22}u^\e_2) (u_2^\e-u_1^\e)- \o{\alpha^\e}(\beta_{11}\o{u}_1 + \beta_{12}\o{u}_2, \beta_{21}\o{u}_1 + \beta_{22}\o{u}_2) (\o{u}_2-\o{u}_1) \right)\phi^\e \psi dx dt \right| =0. 
\end{equation}
\end{lemma}
\begin{proof}
We compute:
\begin{equation}\begin{split}
&\o{\alpha^\e}(\beta_{11}u^\e_1 + \beta_{12}u^\e_2, \beta_{21} u^\e_1 + \beta_{22}u^\e_2) (u_2^\e-u_1^\e)- \o{\alpha^\e}(\beta_{11}\o{u}_1 + \beta_{12}\o{u}_2, \beta_{21}\o{u}_1 + \beta_{22}\o{u}_2) (\o{u}_2-\o{u}_1)=\\
&\o{\alpha^\e}(\beta_{11}u^\e_1 + \beta_{12}u^\e_2, \beta_{21} u^\e_1 + \beta_{22}u^\e_2) (u_2^\e-u_1^\e - \o{u}_2+\o{u}_1)\\
+ &(\o{u}_2-\o{u}_1)(\o{\alpha^\e}(\beta_{11}u^\e_1 + \beta_{12}u^\e_2, \beta_{21} u^\e_1 + \beta_{22}u^\e_2) - \o{\alpha^\e}(\beta_{11}\o{u}_1 + \beta_{12}\o{u}_2, \beta_{21}\o{u}_1 + \beta_{22}\o{u}_2) ), 
\end{split}
\end{equation}
so
\begin{equation}\begin{split}
&\E \left| S^\e_2 \right|\leq C\E \int_0^T\|u_2^\e(t) - \o{u}_2(t) \|_{L^2(D)} + \|u_1^\e(t) - \o{u}_1(t) \|_{L^2(D)} dt, 
\end{split}
\end{equation}
based on the uniform Lipschitz condition of $\o{\alpha^{\e}}$ and the imbedding of $H_0^1(D)$ into $L^2(D)$. The uniform bounds for $u_1^\e, \ u_2^\e$ now give \eqref{convre''}. 
\end{proof}

\subsubsection{Convergence of $S^\e_3$}\label{subs44}
\label{subs44}
\begin{lemma}
\label{lemmaconv1'}
For $\o{u}_1, \o{u}_2 \in L^\infty (\Omega; C([0, T]; L^2(D)))$, $\phi^\e\in H_0^1(D)$ uniformly bounded and $\psi\in C[0, T]$ we define by $S^\e_3$ the integral $$\int_0^T \int_D \left(\o{\alpha^\e}(\beta_{11}\o{u}_1 + \beta_{12}\o{u}_2, \beta_{21}\o{u}_1 + \beta_{22}\o{u}_2)  - \o{\alpha}(\beta_{11}\o{u}_1 + \beta_{12}\o{u}_2, \beta_{21}\o{u}_1 + \beta_{22}\o{u}_2)) (\o{u}_2-\o{u}_1\right)\phi^\e \psi dxdt.$$ Then:
\begin{equation}
\label{convre'}
\begin{split}
\lim_{\e\to 0} \E\left| S^\e_3\right| =& 0. 
\end{split}
\end{equation}
\end{lemma}
\begin{proof}
For fixed $t\in[0, T]$ and $\omega\in\Omega$ we consider the sequence of functions $F^\e_t: L^2(D)^2 \to L^2(D)$,
$$F^\e_t(z) (x) = \left(\alpha \left(\dfrac{x}{\e}, z(x)\right) - \int_Y\alpha \left(y, z(x)\right) \right) (\o{u}_2(t, x) - \o{u}_1(t,x)). $$
We show now that $F^\e_t (z)$ converges in $L^2(D)$ to $0$, for every $z\in L^2(D)^2$. Let $z_n$ and $w_n$ two sequences of continuous functions converging in $L^2(D)^2$ to $z$ and in $L^2(D)$ to $\o{u}_2-\o{u}_1$. We use Lemma 1. 3 from \cite{Allaire} to get that $( F^\e_t)_n(x) = \left(\alpha \left(\dfrac{x}{\e}, z_n(x)\right) - \displaystyle\int_Y\alpha \left(y, z_n(x)\right) \right) w_n(x)$ converges to $0$ in $L^2(D)$. 

But $$\left|( F^\e_t)_n(x) - F^\e_t(z)(x) \right| \leq c |w_n(x)- \o{u}_2(t,x)+\o{u}_1)(t,x)| + c|z_n(x)-z(x)|, $$
based on the Lipschitz condition and boundedness for $\alpha$ so we deduce that $F^\e_t (z)$ converges in $L^2(D)$ to $0$. The sequence is uniformly bounded by $C\|  \o{u}_1 \|_{L^\infty (\Omega; C([0, T]; L^2(D)))}+C\| \o{u}_2 \|_{L^\infty (\Omega; C([0, T]; L^2(D)))}$, Vitali's convergence theorem implies that the sequence of the integrals with respect to the probability measure on $L^2(D)^2$, $\mu=\mu^{(\beta_{11}\o{u}_1(t) + \beta_{12}\o{u}_2(t), \beta_{21}\o{u}_1(t) + \beta_{22}\o{u}_2(t))}$ also converges to $0$ in $L^2(D)$:
$$\lim_{\e\to 0}\int_{L^2(D)^2} F^\e_t (z) d\mu(z) = 0 \ in \ L^2(D), $$
which can be rewritten as
$$\lim_{\e\to 0}\left(\o{\alpha^\e}(\beta_{11}\o{u}_1 + \beta_{12}\o{u}_2, \beta_{21}\o{u}_1 + \beta_{22}\o{u}_2) - \o{\alpha}(\beta_{11}\o{u}_1 + \beta_{12}\o{u}_2, \beta_{21}\o{u}_1 + \beta_{22}\o{u}_2)\right) (\o{u}_2-\o{u}_1)= 0 \ in \ L^2(D). $$
This implies that $\mathbb{P}$ a. s. and for every $t\in[0, T]$
$$\lim_{\e\to 0} \int_D \left(\o{\alpha^\e}(\beta_{11}\o{u}_1 + \beta_{12}\o{u}_2, \beta_{21}\o{u}_1 + \beta_{22}\o{u}_2)  - \o{\alpha}(\beta_{11}\o{u}_1 + \beta_{12}\o{u}_2, \beta_{21}\o{u}_1 + \beta_{22}\o{u}_2)) (\o{u}_2-\o{u}_1\right)\phi^\e \psi dx = 0, $$
with the sequence being also uniformly bounded. We apply the bounded convergence theorem and integrate over $\Omega \times [0, T]$ to get the result.
\end{proof}

Now, we are able to pass to the limit on the remaining terms of the variational equation \eqref{weaksole1}. The uniform bounds \eqref{est2'} and \eqref{est3'} hold for $u^\e_1$ and $u^\e_2$. So the sequences are a. s. $\omega\in\Omega$ contained in a compact set $\mathcal{K}$ of $w\mbox{-}L^2(0, T;H_0^1(D))$ so they are tight in $w\mbox{-}L^2(0, T;H_0^1(D))\cap C([0, T];L^2(D)) $. Then, there exist subsequences $u^{\e'}_1, \ u^{\e'}_2$ and random elements $\o{u}_1, \ \o{u}_2 \in L^2(0, T;H_0^1(D))\cap C([0, T];L^2(D))$ such that $u^{\e'}_1,\ u^{\e'}_2$ converge in distribution to $\o{u}_1,\ \o{u}_2$ in $w\mbox{-}L^2(0, T;H_0^1(D))\cap C([0, T];L^2(D))$. Skorokhod theorem gives us the existence of subsequences $u^{\e''}_1, \ u^{\e''}_2$ and  $\wt{u^{\e''}_1}, \ \wt{u^{\e''}_2}$ with the same distribution as $u^{\e''}_1,\ u^{\e''}_2$ defined on another probability space $\wt{\Omega}$ that converges pointwise to $\wt{\o{u}}_1, \ \wt{\o{u}}_2$ with the same distribution as $\o{u}_1, \ \o{u}_2$. It follows from here that $\wt{\o{u}}_1,\ \wt{\o{u}}_2 \in \mathcal{K}$ a. s. so $\wt{\o{u}_1},\ \wt{\o{u}_2}\in L^\infty(\wt{\Omega}, L^2(0, T;H_0^1(D))))$ and $\o{u}_1, \ \o{u}_2 \in L^\infty(\Omega, L^2(0, T;H_0^1(D)))$. 

In the variational formulation \eqref{weaksole1} for $u^{\e''}_1$ we use a test function $\phi^{\e''}=\phi + \e'' \nabla \phi \cdot \chi^{*\e''}_1$ where $\phi \in C_0^\infty(D)$, multiply it with $\psi'$ where $\psi \in C_0^1(0, T)$ to get:
\begin{equation}
\label{weaksole'}
\begin{split}
&\int_D u_1^\e(t) \phi^{\e''}\psi' dx - \int_D u^\e_{01} \phi^{\e''}\psi' dx + \int_0^t \int_D A_1^\e\nabla u_1^\e (s) \nabla \phi^{\e''}\psi' dx ds =\\
&\int_0^t \int_D \alpha^\e(v_1^\e(s), v_2^\e(s)) (u_2^\e(s)-u_1^\e(s)) \phi^{\e''}\psi' dx ds + \int_0^t \int_D f_1(s) \phi^{\e''}\psi' dx ds,
\end{split}
\end{equation}
We notice that
\begin{equation}
\label{convre'''}
\lim_{\e''\to 0} \E\left|\displaystyle\int_0^T \int_D \left(\alpha^\e(v_1^\e(t), v_2^\e(t)) (u_2^\e(t)-u_1^\e(t))   - \o{\alpha}(\beta_{11}\o{u}_1(t)+\beta_{12}\o{u}_2(t),\beta_{21}\o{u}_1(t)+\beta_{22}\o{u}_2(t)\right)\phi^{\e''} \psi(t)dx dt \right| = 0,
\end{equation}
rewriting the integral as $ S^{\e''}_1 + S^{\e''}_2 +S^{\e''}_3$ and using the convergences given by  
Lemmas \ref{lemmaconv1}, \ref{lemmaconv1''} and \ref{lemmaconv1'}. We obtain that:
\begin{equation}\label{eq6}
\begin{split}
\lim_{\e''\to 0} \E &\left|\int_0^T \int_D u^{\e''}_1(t) \phi^{\e''} \psi'(t) dx dt - \int_0^T\int_D u_{01} \phi^{\e''} \psi'(t) dx dt - \int_0^T \int_D A^{\e''}_1 \nabla u^{\e''}_1(t)\nabla\phi^{\e''}\psi(t)dx dt+\right. \\
&\left. \int_0^T \int_D\o{\alpha}(\beta_{11}\o{u}_1(t)+\beta_{12}\o{u}_2(t),\beta_{21}\o{u}_1(t)+\beta_{22}\o{u}_2(t))(\o{u}_2(t)-\o{u}_1(t))\phi^{\e''} \psi(t)dx dt+ \int_0^T \int_D f_1(t)\phi^{\e''}\psi(t) dxdt\right|=\\
\lim_{\e''\to 0} \wt\E &\left|\int_0^T \int_D \wt{u^{\e''}_1}(t) \phi^{\e''} \psi'(t) dx dt - \int_0^T\int_D u_0 \phi \psi'(t) dx dt - \int_0^T \int_D A^{\e''}_1 \nabla \wt{u^{\e''}_1}(t)\nabla\phi^{\e''}\psi(t)dx dt+\right. \\
&\left. \int_0^T \int_D\o{\alpha}(\beta_{11}\o{u}_1(t)+\beta_{12}\o{u}_2(t),\beta_{21}\o{u}_1(t)+\beta_{22}\o{u}_2(t))(\o{u}_2(t)-\o{u}_1(t))\phi\psi(t)dx dt+ \int_0^T \int_D f_1(t)\phi\psi(t) dxdt\right|=0. 
\end{split}
\end{equation}
We make now several calculations under the integral in the above equation and then pass to the limit pointswise in $\wt{\omega} \in \wt{\Omega}$:
\begin{equation*}
\begin{split}
& \int_0^T \int_D A^{\e''}_1 \nabla \wt{u^{\e''}_1}\nabla\left(\phi +\e''\nabla\phi\cdot \chi^{*\e''}_1 \right)\psi(t)dxdt= \\
& \int_0^T \int_D A^{\e''}_1 \nabla \wt{u^{\e''}_1}\left(\nabla\phi +\e''\nabla\nabla\phi \chi^{*\e''}_1 +\e'' \nabla\phi \nabla\chi^{\e''}_1\right)\psi(t)dxdt= \\
& \int_0^T \int_D A^{\e''}_1 \nabla \wt{u^{\e''}_1}\nabla\phi\psi(t) +\e''A^{\e''}_1 \nabla \wt{u^{\e''}_1}\nabla\nabla\phi \chi^{*\e''}_1\psi(t) +\e'' A^{\e''}_1 \nabla \wt{u^{\e''}_1}\nabla\phi \nabla\chi^{\e''*}_1\psi(t)dxdt= \\
& \int_0^T \int_D A^{\e''}_1 \nabla \wt{u^{\e''}_1}\nabla\phi\psi(t) +\e''A^{\e''}_1 \nabla \wt{u^{\e''}_1}\nabla\nabla\phi \chi^{*\e''}_1\psi(t) +\e'' A^{\e''}\nabla\chi^{\e''}_1 \nabla \wt{u^{\e''}_1}\nabla\phi\psi(t)dxdt. 
\end{split}
\end{equation*}
From the equation \eqref{cellpre1} satisfied by $\chi^{\e''}_1$ we have that
\begin{equation*}
\begin{split}
\int_D A^{\e''}_1\left( I+ \e'' \nabla \chi^{\e''}_1 \right) \nabla \left(\wt{u^{\e''}_1} \nabla \phi \right) dx&=0\Rightarrow\\
\int_D \left(A^{\e''}_1 \nabla \wt{u^{\e''}_1} \nabla \phi +\e'' A^{\e''}_1 \nabla\chi^{\e''}_1 \nabla \wt{u^{\e''}_1} \nabla \phi \right)dx&=-\int_D A^{\e''}_1 \wt{u^{\e''}_1} \nabla\nabla \phi dx - \int_D \e'' A^{\e''}_1\nabla\chi^{\e''}_1 \wt{u^{\e''}_1} \nabla\nabla\phi dx, 
\end{split}
\end{equation*}
so we get that
\begin{equation*}
\begin{split}
& \int_0^T \int_D A^{\e''}_1 \nabla \wt{u^{\e''}_1}\nabla\left(\phi +\e''\nabla\phi\cdot \chi^{*\e''}_1 \right)\psi(t)dxdt= \\
& \int_0^T \int_D \left(\e''A^{\e''}_1 \nabla \wt{u^{\e''}_1}\nabla\nabla\phi \chi^{*\e''}_1\psi(t)- A^{\e''}_1 \wt{u^{\e''}} \nabla\nabla \phi \psi(t)- \e'' A^{\e''}_1\nabla\chi^{\e''}_1 \wt{u^{\e''}} \nabla\nabla\phi\psi(t)\right) dxdt=\\
& \int_0^T \int_D \left(\e''A^{\e''}_1 \nabla \wt{u^{\e''}_1}\nabla\nabla\phi \chi^{*\e''}_1\psi(t)- A^{\e''}_1\left(I + \e''\nabla\chi^{\e''}_1\right) \wt{u^{\e''}_1} \nabla\nabla \phi \psi(t) \right)dxdt, 
\end{split}
\end{equation*}
and will converge pointwise in $\wt{\Omega}$ (see \cite{Allaire} Lemma 1. 3) to
$$ \int_0^T \int_D - \o{A}_1 \wt{\o{u}_1} \nabla\nabla \phi \psi(t)dxdt = \int_0^T \int_D \o{A}_1 \nabla \wt{\o{u}_1} \nabla \phi \psi(t)dxdt. $$
The sequence given in \eqref{eq6} above converges in $L^1(\wt{\Omega})$ to $0$ and pointwise to 

\begin{equation*}
\begin{split}
&\int_0^T \int_D \left( \wt{\o{u}_1}(t)\phi\psi'(t) - u_{01} \phi\psi' (t) -\o{A}_1 \nabla \wt{\o{u}_1}\nabla\phi\psi(t) 
+ f_1(t)\phi\psi(t) \right) dxdt \\
&+ \int_0^T \int_D \left( \o{\alpha} (\beta_{11}\wt{\o{u}_1}(t)+\beta_{12}\wt{\o{u}_2}(t),\beta_{21}\wt{\o{u}_1}(t)+\beta_{22}\wt{\o{u}_2}(t)) (\wt{\o{u}_2}(t)-\wt{\o{u}_1}(t)) \phi\psi(t)\right)dx dt,
\end{split}
\end{equation*}

and similarly is true for $\wt{\o{u}_2}$ which means that $\wt{\o{u}_1},\wt{\o{u}_2}$ is pointwise the weak solution of the deterministic system \eqref{eqou1} which, according to Theorem \ref{thexunou} has a unique solution, so $\wt{\o{u}_1}, \wt{\o{u}_2}$ and $\o{u}_1, \o{u}_2$ are deterministic. Then, the whole sequences $u^{\e''}_1, u^{\e''}_2$ converge to $\o{u}_1, \o{u}_2$ in distribution, and since the limits are deterministic then the convergence is also in probability see \cite{JP} Theorem 18.3. 

\subsection{Proof of theorem \ref{thexun}}

To prove the existence of solutions, we will follow a similar method used previously in \cite{BEM_19},  through a Galerkin approximation procedure. We consider $(e_k)_{k\geq 1}$ a sequence of linearly independent elements in $H_0^1(D)\cap L^\infty(D)$ such that $span\{e_k\ | \ k\geq 1\}$ is dense in $H_0^1(D)$. We define the $n$-dimensional space $H_0^1(D)_n$ for every $n>0$ as $span\{e_k\ | \ 1\leq k \leq n\}$ and we denote by $\Pi_n$ the projection operator from $L^2(D)$ onto $H_0^1(D)_n$. 

Let us denote by $w_i^\e(t), \ 1\leq i \leq 2$ the following processes
\begin{equation}\label{w}
w_i^\e(t)=e^{-t/\e}v_{0i}^\e+\frac{\sqrt{Q_i}}{\sqrt{\e}} \int_0^t e^{-(t-s)/\e} dW_i(s) \in L^2(\Omega;C([0, T];L^2(D)). 
\end{equation}

Now, in order to prove the existence of solutions, we define the Galerkin approximation 
$$(u^\e_{1n}(t, \omega), z^\e_{1n}(t, \omega), u^\e_{2n}(t, \omega), z^\e_{2n}(t, \omega) )\in H_0^1(D)_n^4$$ 
a. s. $\omega\in\Omega$, solution of the following system

\begin{equation}
\label{weaksolen}
\begin{split}
\int_D\frac{\partial u^\e_{1n}} {\partial t} (t) \phi dx &+ \int_D A_1^\e \nabla u_1^\e (t) \nabla \phi dx = \int_D f_1(t) \phi dx+\\ 
&\int_D \alpha^\e(z^\e_{1n}(t)+w_1^\e(t), z^\e_{2n}(t)+w_2^\e(t)) (u_{2n}^\e(t) -  u_{1n}^\e(t) )\phi dx \\
\int_D\frac{\partial u^\e_{2n}} {\partial t} (t) \phi dx &+ \int_D A_2^\e \nabla u_2^\e (t) \nabla \phi dx =\int_D f_2(t) \phi dx \\
&\int_D \alpha^\e(z^\e_{1n}(t)+w_1^\e(t), z^\e_{2n}(t)+w_2^\e(t)) (u_{1n}^\e(t) -  u_{2n}^\e(t) )\phi dx ,
\end{split}
\end{equation}
for every $\phi \in H_0^1(D)_n$, $u^\e_{1n}(0, \omega) =\Pi_n u^\e_{01}, \ u^\e_{2n}(0, \omega) =\Pi_n u^\e_{02}$, 
\begin{equation}
\label{mildsolen}
\begin{split}
\frac{\partial z^\e_{1n}}{\partial t}(t) = -\frac{1}{\e}(z^\e_{1n}(t)-\beta_{11}u^\e_{1n}(t)-\beta_{12}u^\e_{2n}(t)), \quad z^\e_{1n}(0)=0,\\
\frac{\partial z^\e_{2n}}{\partial t}(t) = -\frac{1}{\e}(z^\e_{2n}(t)-\beta_{21}u^\e_{1n}(t)-\beta_{22}u^\e_{2n}(t)), \quad z^\e_{2n}(0)=0, 
\end{split}
\end{equation}

Then, we pass to the limit on $(u^\e_{1n}, z^\e_{1n} , u^\e_{2n}, z^\e_{2n} )$ when $n \to \infty$. 

We write $$u_{1n}^\e(\omega, t, x) = \sum_{k=1}^n a^\e_{1k}(\omega, t) e_k(x),\   u_{2n}^\e(\omega, t, x) = \sum_{k=1}^n a^\e_{2k}(\omega, t) e_k(x)$$ and 
$$z_{1n}^\e(\omega, t, x) = \sum_{k=1}^n b^\e_{1k}(\omega, t) e_k(x), \ z_{2n}^\e(\omega, t, x) = \sum_{k=1}^n b^\e_{2k}(\omega, t) e_k(x)$$ 
Then, we make the following notations: $$b_{ij}=\displaystyle\int_D e_i(x) e_j(x) dx, \ c^\e_{1ij}=\displaystyle\ \int_D \sum_{p=1}^{ n}\sum_{ q=1}^{ n}a_{1pq}\left(\dfrac{x}{\e}\right)\dfrac{\partial e_i}{\partial x_q} \dfrac{\partial e_j}{\partial x_p} dx,$$ $$c^\e_{2ij}=\displaystyle\ \int_D \sum_{p=1}^{ n}\sum_{ q=1}^{ n}a_{2pq}\left(\dfrac{x}{\e}\right)\dfrac{\partial e_i}{\partial x_q} \dfrac{\partial e_j}{\partial x_p} dx, \ f_{1j}(s) = \int_D f_1(s, x) e_j(x) dx, \ f_{2j}(s) = \int_D f_2(s, x) e_j(x) dx, $$
where $a_{1pq}$ and $a_{2pq}$ are the entries of respectively the matrices $A_1$ and $A_2$ defined previously in Section \ref{sec2}. 
Moreover, we set
$$(F^\e_n)_{ij}(\omega, t, b_{11}, ... b_{1n}, b_{21},... b_{2n})= \int_D  \alpha^\e\left(w^\e_1(t)+\sum_{k=1}^n b_{1k}^\e(t) e_k, w^\e_2(t)+\sum_{k=1}^n b_{2k}^\e(t) e_k\right) e_i e_jdx$$
and get the following system:

\begin{equation}
\label{weaksolen''}
\left\{
\begin{array}{rll}
& \displaystyle\sum_{k=1}^n\frac{\partial a^\e_{1k}}{\partial t} b_{kl}+\sum_{k=1}^n a^\e_{1k} c^\e_{1kl}- \sum_{k=1}^n (a^\e_{2k}-a^\e_{1k}) (F^\e_n)_{kl}(b_{11}^\e, ..., b_{2n}^\e)= f_{1l}(t), \\
\\
& \displaystyle\sum_{k=1}^n\frac{\partial a^\e_{2k}}{\partial t} b_{kl}+\sum_{k=1}^n a^\e_{2k} c^\e_{2kl}- \sum_{k=1}^n (a^\e_{1k}-a^\e_{2k}) (F^\e_n)_{kl}(b_{11}^\e, ..., b_{2n}^\e)= f_{2l}(t), \\
\\
& \displaystyle\frac{\partial b^\e_{1k}}{\partial t}=-\dfrac{1}{\e}\left( b^\e_{1k} - - \beta_{11}a_{1k}^\e-\beta_{12}a_{2k}  \right), \ 1\leq k \leq n\\
\\
& \displaystyle\frac{\partial b^\e_{2k}}{\partial t}=-\dfrac{1}{\e}\left( b^\e_{2k} - - \beta_{21}a_{1k}^\e-\beta_{22}a_{2k}  \right), \ 1\leq k \leq n\\
\\
&a_{1k}^\e(0)=\displaystyle\int_D u^\e_{01} e_k dx, \ 1\leq k \leq n\\
\\
&a_{2k}^\e(0)=\displaystyle\int_D u^\e_{02} e_k dx, \ 1\leq k \leq n\\
\\
&b_{1k}^\e(0)=0, \ 1\leq k \leq n\\
\\
&b_{2k}^\e(0)=0, \ 1\leq k \leq n
\end{array}
\right. 
\end{equation}
for each $1\leq l \leq n$. 
Given the linearly independence of the sequence $(e_{k})_{k\geq 1}$, the form of the functions $(F^\e_n)_{ij}$ and the Lipschitz condition satisfied by $\alpha$, the system has for every $T>0$ an unique $\mathcal{F}_t$ - measurable solution $(a^\e_{1k})_{1\leq k \leq n}, \ (a^\e_{2k})_{1\leq k \leq n}, \ (b^\e_{1k})_{1\leq k \leq n}, \ (b^\e_{2k})_{1\leq k \leq n} \in C([0, T];L^\infty(\Omega))$, with $(a^\e_{1k})_{1\leq k \leq n}, \ (a^\e_{2k})_{1\leq k \leq n}, \ (b^\e_{1k})_{1\leq k \leq n}, \ (b^\e_{2k})_{1\leq k \leq n}  \in W^{1, 2}(0, T)$ a. s. $\omega\in\Omega$. This means that $u_{1n}^\e,\ u_{2n}^\e$, $z_{1n}^\e=v_{1n}^\e-w_1^\e, \ z_{2n}^\e=v_{2n}^\e-w_2^\e$ are a. s. a solution for:
\begin{equation}
\label{weaksolen'''}
\left\{
\begin{array}{rll}
&\displaystyle\int_D\frac{\partial u^\e_{1n}} {\partial t} (t) \phi dx + \int_D A_1^\e \nabla u_{1n}^\e (t) \nabla \phi dx -\\
&\int_D \alpha^\e(z^\e_{1n}(t)+w_1^\e(t), z^\e_{2n}(t)+w_2^\e(t)) (u_{2n}^\e(t) - u_{1n}^\e(t)) \phi dx
= \displaystyle\int_D f_1(t) \phi dx, \\
\\
&\displaystyle\int_D\frac{\partial u^\e_{2n}} {\partial t} (t) \phi dx + \int_D A_2^\e \nabla u_{2n}^\e (t) \nabla \phi dx - \\
&\int_D \alpha^\e(z^\e_{1n}(t)+w_1^\e(t), z^\e_{2n}(t)+w_2^\e(t)) (u_{1n}^\e(t) - u_{2n}^\e(t)) \phi dx = \displaystyle\int_D f_2(t) \phi dx, \\
&d z_{1n}^\e =-\dfrac{1}{\e}\left( z^\e_{1n} -  - \beta_{11}u_{1n}^\e-\beta_{12}u_{2n}^\e  \right), \\
\\
&d z_{2n}^\e =-\dfrac{1}{\e}\left( z^\e_{2n} - \beta_{21}u_{1n}^\e-\beta_{22}u_{2n}^\e \right), \\
\\
&u_{1n}^\e(0)=\Pi_n u^\e_{01}, \\
\\
&u_{2n}^\e(0)=\Pi_n u^\e_{02}, \\
\\
&z_{1n}^\e(0)=0, \\
\\
&z_{2n}^\e(0)=0, 
\end{array}
\right. 
\end{equation}
for every $\phi\in H_0^1(D)_n$. We take $\phi= u_{1n}^\e$ in the first equation of \eqref{weaksolen'''}, and  $\phi= u_{2n}^\e$ in the second equation to derive that a. s. $\omega\in\Omega$ :

\begin{equation}\nonumber
\begin{split}
\frac{\partial}{\partial t} \|u^\e_{1n} \|^2_{L^2(D)} &\leq \|f_1(t)\|^2_{L^2(D)} + C \|u^\e_{1n} \|^2_{L^2(D)} + C \|u^\e_{2n} \|^2_{L^2(D)} , \\
\frac{\partial}{\partial t} \|u^\e_{2n} \|^2_{L^2(D)} &\leq \|f_2(t)\|^2_{L^2(D)} + C \|u^\e_{1n} \|^2_{L^2(D)} + C \|u^\e_{2n} \|^2_{L^2(D)} , \Rightarrow\\
 \|u^\e_{1n} \|^2_{L^2(D)} +  \|u^\e_{2n} \|^2_{L^2(D)} &\leq e^{Ct} \left( \|f_1\|_{L^2(0, T;L^2(D))} + \|f_2\|_{L^2(0, T;L^2(D))} +  \|u^\e_{01}\|_{L^2(D)} + \|u^\e_{02}\|_{L^2(D)} \right), 
\end{split}
\end{equation}
so
\begin{equation}
\label{estune1}
\sup_{n>0}\| u_{1n}^\e\|_{L^\infty(0, T;L^2(D))} ,\ \sup_{n>0}\| u_{2n}^\e\|_{L^\infty(0, T;L^2(D))}  \leq C_T(1+\|u^\e_{01}\|_{L^2(D)} + \|u^\e_{02}\|_{L^2(D)} ). 
\end{equation}
We also obtain based on the positivity of $A_1$ that
\begin{equation}\nonumber
\begin{split}
\int_0^T m\|\nabla u^\e_{1n} \|^2_{L^2(D)^3} ds +\frac{1}{2} \| u_{1n}^\e(T)\|^2_{L^2(D)}&\leq \int_0^T \int_D f_1(t) u_{1n}^\e dxdt + \\
\frac{1}{2}\|u^\e_{01} \|^2_{L^2(D)}&+\int_0^T C(\| u^\e_{1n} \|^2_{L^2(D)} +\| u^\e_{2n} \|^2_{L^2(D)})ds
\end{split}
\end{equation}
which infers that

\begin{equation*}
%\begin{split}
\int_0^T m \|\nabla u^\e_{1n} \|^2_{L^2(D)^3} ds\leq T \|f_1\|_{L^2(0, T;L^2(D))}\| u_{1n}^\e\|_{L^\infty(0, T;L^2(D))}+ C_T(1+\|u^\e_{01}\|_{L^2(D)} +\|u^\e_{02}\|_{L^2(D)})).
%\end{split}
\end{equation*}
Hence, 

\begin{equation}
\label{estune2}
\sup_{n>0}\| u_{1n}^\e\|_{L^2(0, T;H_0^1(D))} , \ \sup_{n>0}\| u_{2n}^\e\|_{L^2(0, T;H_0^1(D))}  \leq C_T(1+\|u^\e_{01}\|_{L^2(D)} + \|u^\e_{02}\|_{L^2(D)} ). 
\end{equation}
The estimates \eqref{estune1} and \eqref{estune2} imply using the first equation of the system \eqref{weaksolen'''} that
\begin{equation}
\label{estune3}
\sup_{n>0}\left\|\dfrac{\partial u_{1n}^\e}{\partial t}\right\|_{L^2(0, T;(H_0^1(D)_n)')}, \ \sup_{n>0}\left\|\dfrac{\partial u_{2n}^\e}{\partial t}\right\|_{L^2(0, T;(H_0^1(D)_n)')}  \leq C_T(1+\|u^\e_{01}\|_{L^2(D)} + \|u^\e_{02}\|_{L^2(D)} ). 
\end{equation}

This means that the sequences $u_{1n}^\e, \ u_{2n}^\e$ are bounded in $L^2(0, T;H_0^1(D))\cap W^{1, 2}(0, T;H^{-1}(D))$ which is compactly embedded in $L^2(0, T;L^2(D))$ (Theorem 2. 1, page 271 from \cite{temam}) and in $C([0, T], H^{-1}(D))$. Hence, there exists 
subsequences $u_{1n'}^\e, \ u_{2n'}^\e$ that converge $P$-a.s. in $L^2(0, T;L^2(D))\cap C([0, T], H^{-1}(D))$ to some $u_1^\e, \ u_2^\e$ which are also weak limits in $L^2(0, T;H_0^1(D)) \cap W^{1, 2}(0, T;H^{-1}(D))$ and weak$^*$ limits in $L^\infty(0, T;L^2(D))$. So using Lemma 1.2, page 260 from \cite{temam} a. s. $\omega\in\Omega$, 
$u_1^\e, \ u_2^\e \in L^2(0, T;H_0^1(D))\cap C ([0, T];L^2(D))\cap W^{1, 2}(0, T;H^{-1}(D))$. 

We also have from \eqref{weaksolen'''} that
$$z^\e_{1n'}(t)=\displaystyle\frac{1}{\e}\int_0^t e^{-(t-s)/\e} u^\e_{1n'}(s) ds$$
will converge $P$-a.s. to $z_1^\e(t)=\displaystyle\frac{1}{\e}\int_0^t e^{-(t-s)/\e} u_1^\e(s) ds$ in $C([0, T];L^2(D))$ and $$z^\e_{2n'}(t)=\displaystyle\frac{1}{\e}\int_0^t e^{-(t-s)/\e} u^\e_{2n'}(s) ds$$
 to $z_2^\e(t)=\displaystyle\frac{1}{\e}\int_0^t e^{-(t-s)/\e} u_2^\e(s) ds$ in $C([0, T];L^2(D))$. 
 
 We remark that the sequences $u^{\e}_{1n'},\ u^{\e}_{2n'}$ are $\mathcal{F}_t$ - measurable in $H^{-1}(D)$.
 
We now pass to the limit when $n'\to\infty$ in the first equation of the system \eqref{weaksolen'''}
pointwise in $\omega\in\Omega$ using the convergences of the sequences $u^\e_{n'}$ and $\dfrac{\partial u^\e_{n'}}{\partial t}$:
$$\lim_{n'\to\infty} \int_0^t \int_D\frac{\partial u^\e_{1n'}} {\partial t} \phi dx ds= \int_0^t\int_D\frac{\partial u^\e_1} {\partial t} \phi dxds$$
and
$$\lim_{n'\to\infty} \int_0^t\int_D A_1^\e \nabla u^\e_{1n'} \nabla \phi dx ds= \int_0^t\int_D A_1^\e \nabla u^\e_1 \nabla \phi dx ds. $$

Also
\begin{equation}\nonumber
\begin{split}
\left|\int_0^t \int_D  \alpha^\e(z^\e_{1n'}+w^\e_{1}, z^\e_{2n'}+w^\e_{2}) (u^\e_{2n'}-u^\e_{1n'}) \phi dxds- 
\int_0^t\int_D  \alpha^\e(z^\e_1+w^\e_1, z^\e_2+w^\e_2) (u^\e_2-u^\e_1 ) \phi dxds\right| &\leq \\
\left| \int_0^t\int_D \alpha^\e(z^\e_1+w^\e_1, z^\e_2+w^\e_2) (u^\e_2-u^\e_1-u^\e_{2n'}+u^\e_{1n'}) \phi dxds \right|
&+\\
 \left|\int_0^t\int_D \left(\alpha^\e(z^\e_{1n'}+w^\e_{1}, z^\e_{2n'}+w^\e_{1}) -\alpha^\e(z^\e_1+w^\e_1, z^\e_2+w^\e_1)\right)  (u^\e_{2n'}-u^\e_{1n'}) \phi dxds\right| &\leq \\
C \left(\int_0^t\int_D ( u^\e_{n'}- u^\e)^2 dxds\right)^{1/2}+ C \int_0^t\int_D \left| z^\e_{n'} - z^\e\right| | u^\e| |\phi | dxds& \leq \\
C \| u^\e_{1n'}- u^\e_1\|_{L^2(0, T;L^2(D))} + C \| u^\e_{2n'}- u^\e_2\|_{L^2(0, T;L^2(D))} &+\\
C \int_0^T \left(\| z^\e_{2n'} - z^\e_2 \|_{L^2(D)} + \| z^\e_{1n'} - z^\e_1 \|_{L^2(D)} \| \right)\left(\|u_1^\e \|_{L^2(D)} + \|u_2^\e \|_{L^2(D)} \right)\| \phi \|_{L^\infty(D)}ds, 
\end{split}
\end{equation}
so we obtain that a. s. 
$$\lim_{n'\to\infty} \int_0^t \int_D  \alpha^\e(z^\e_{1n'}+w^\e_{1}, z^\e_{2n'}+w^\e_{2}) (u^\e_{2n'}-u^\e_{1n'}) \phi dxds = 
\int_0^t\int_D \alpha^\e(z^\e_{1}+w^\e_{1}, z^\e_{2}+w^\e_{2}) (u^\e_{2}-u^\e_{1}) \phi dxds. $$ 

Using these convergences, we obtain in the limit:

\begin{equation}
\label{weaksolen''''}
\left\{
\begin{array}{rll}
&\displaystyle\int_0^t \int_D\left(\frac{\partial u^\e_1} {\partial t} + A_1^\e \nabla u_1^\e\nabla -\alpha^\e(z^\e_1+ w^\e_1, z^\e_2+ w^\e_2)( u_2^\e - u_1^\e )\right)\phi dxds = \displaystyle \int_0^t \displaystyle\int_D f_1 \phi dxds, \\
\\
&\displaystyle\int_0^t \int_D\left(\frac{\partial u^\e_2} {\partial t} +  A_2^\e \nabla u_1^\e\nabla -  \alpha^\e( z^\e_1+ w^\e_1, z^\e_2+ w^\e_2)( u_1^\e - u_2^\e )\right)\phi dxds = \displaystyle \int_0^t \displaystyle\int_D f_2 \phi dxds, \\
\\
&d z_1^\e =-\dfrac{1}{\e}\left( z_1^\e - \beta_{11}u_1^\e -\beta_{12}u_2^\e  \right), \\
\\
&d z_2^\e =-\dfrac{1}{\e}\left( z_2^\e - \beta_{21}u_1^\e -\beta_{22}u_2^\e  \right), \\
\\
& u^\e_1(0)=u^\e_{01}, \\
\\
& u^\e_2(0)=u^\e_{02}, \\
\\
& z^\e_1(0)=0, \\
\\
& z^\e_2(0)=0, 
\end{array}
\right. 
\end{equation}
pointwise in $\omega \in\Omega$ for every $\phi\in H_0^1(D)_n$, so by density it is true for any $\phi \in H_0^1(D)$. Now, let us denote
\begin{equation}
\label{vne}
\begin{split}
v^\e_{1n}(t)=z^\e_{1n} (t)- w_1^\e(t), \ v^\e_{2n}(t)=z^\e_{2n} (t)- w_2^\e(t), \\
v^\e_{1}(t)=z^\e_{1} (t)- w_1^\e(t), \ v^\e_{2}(t)=z^\e_{2} (t)- w_2^\e(t),
\end{split}
\end{equation}
then we deduce that $(u^\e, v^\e)$ is a solution for our initial system in the sense given by \eqref{weaksole1}, \eqref{weaksole2} and \eqref{mildsole}. The solution $( u^\e, v^\e)$ is $\mathcal{F}_t$ - measurable as the limit of the Galerkin approximation $(u^\e_{n'}, v^\e_{n'})$ which is $\mathcal{F}_t$ - measurable by construction. 
Furthermore, given the uniform estimates for $u^\e_0$ it is easy to obtain from \eqref{estune1}--\eqref{estune3} the estimates \eqref{est1}--\eqref{est3} and \eqref{est4} follows from the uniform bounds for $v^\e_0$. 

Now, we prove the uniqueness. Let us assume that we have two solutions $\{u^\e_{11}, u^\e_{12}, v^\e_{11}, v^\e_{12}\}$ and $\{u^\e_{21}, u^\e_{22}, v^\e_{21}, v^\e_{22}\}$ for the system. Then, 
\begin{equation*}
\begin{split}
&\int_D (u^\e_{21}(t) -u^\e_{11}(t)) \phi dx + \int_0^t \int_D A_1^\e(\nabla u^\e_{21} -\nabla u^\e_{11}) \nabla \phi dx ds =\\
&\int_0^t \int_D (\alpha^\e(v^\e_{21}, v^\e_{22}) (u^\e_{22}-u^\e_{21}) -\alpha^\e(v^\e_{11}, v^\e_{12}) (u^\e_{12}-u^\e_{11}))\phi dx ds, 
\end{split}
\end{equation*}
and
\begin{equation}
\nonumber
v^\e_{21}(t)-v^\e_{11}(t) = \frac{1}{\e} \int_0^t \left(\beta_{11}(u^\e_{21}(s)-u^\e_{11}(s)) + \beta_{12}(u^\e_{22}(s)-u^\e_{12}(s))\right)  e^{-(t-s)/\e} ds. 
\end{equation}
we take $\phi = u^\e_{21} - u^\e_{11}$ and we get:
\begin{equation}
\nonumber
\begin{split}
&\int_D (u^\e_{21}(t) -u^\e_{11}(t))^2 dx + \int_0^t \int_D A_1^\e(\nabla u^\e_{21} -\nabla u^\e_{11})^2 dx ds =\\
&\int_0^t \int_D \alpha^\e(v^\e_{21}, v^\e_{22}) (u^\e_{22} - u^\e_{21}-u^\e_{12} +u^\e_{11}) (u^\e_{21}-u^\e_{11})dx ds +\\
&\int_0^t \int_D (\alpha^\e(v^\e_{21}, v^\e_{22})-\alpha^\e(v^\e_{11}, v^\e_{12}) )(u^\e_{12}-u^\e_{11} )(u^\e_{21} - u^\e_{11})dx ds\leq \\
&c\int_0^t \|u^\e_{21} - u^\e_{11}\|_{L^2(D)}^2ds +c\int_0^t \|u^\e_{22} - u^\e_{12}\|_{L^2(D)}^2ds+ \\
&c\int_0^t \int_D (|v^\e_{21}-v^\e_{11}|+|v^\e_{22}-v^\e_{12}|) |u^\e_{12}-u^\e_{11} | u^\e_{21}-u^\e_{11}| dx ds \leq \\
&c\int_0^t \|u^\e_{21} - u^\e_{11}\|_{L^2(D)}^2ds +c\int_0^t \|u^\e_{22} - u^\e_{12}\|_{L^2(D)}^2ds+\\
& c \int_0^t (\|v^\e_{21}-v^\e_{11}\|_{L^2(D)} + \|v^\e_{22}-v^\e_{12}\|_{L^2(D)})\|u^\e_{12}-u^\e_{11} \|_{L^4(D)}\|u^\e_{21} - u^\e_{11} \|_{L^4(D)}ds \leq \\
&c\int_0^t \|u^\e_{21} - u^\e_{11}\|_{L^2(D)}^2ds +c\int_0^t \|u^\e_{22} - u^\e_{12}\|_{L^2(D)}^2ds + \\
&c \left(\int_0^t  (\|v^\e_{21}-v^\e_{11}\|^2_{L^2(D)} + \|v^\e_{22}-v^\e_{12}\|^2_{L^2(D)}) \|u^\e_{12}-u^\e_{11} \|^2_{L^4(D)}ds \right)^{1/2} \left(\int_0^t \| u^\e_{21}-u^\e_{11} \|^2_{L^4(D)}ds\right)^{1/2}\leq \\
&c\int_0^t \|u^\e_{21} - u^\e_{11}\|_{L^2(D)}^2ds +c\int_0^t \|u^\e_{22} - u^\e_{12}\|_{L^2(D)}^2ds +\\
&c \int_0^t(\|v^\e_{21}-v^\e_{11}\|^2_{L^2(D)} + \|v^\e_{22}-v^\e_{12}\|^2_{L^2(D)}) \|\nabla u^\e_{12}-\nabla u^\e_{11} \|^2_{L^2(D)^3}ds +\dfrac{m}{2} \int_0^t \|\nabla u^\e_{21} - \nabla u^\e_{11} \|^2_{L^2(D)^3}ds\leq\\
&c\int_0^t \|u^\e_{21} - u^\e_{11}\|_{L^2(D)}^2ds +c\int_0^t \|u^\e_{22} - u^\e_{12}\|_{L^2(D)}^2ds+ \\
&c \int_0^t (\|v^\e_{21}-v^\e_{11}\|^2_{L^2(D)} + \|v^\e_{22}-v^\e_{12}\|^2_{L^2(D)}) \|\nabla u^\e_{12}-\nabla u^\e_{11}  \|^2_{L^2(D)^3}ds +\dfrac{m}{2} \int_0^t \|\nabla u^\e_{21} - \nabla u^\e_{11} \|^2_{L^2(D)^3}ds, 
\end{split}
\end{equation}
where we used H\"{o}lder's inequality, the imbedding of $H_0^1(D)$ into $L^4(D)$ and the Lipschitz condition of $\alpha$.  So:
\begin{equation}\nonumber
\begin{split}
&\int_D (u^\e_{21}(t) -u^\e_{11}(t))^2 dx  \leq c\int_0^t \|u^\e_{21} - u^\e_{11}\|_{L^2(D)}^2ds +c\int_0^t \|u^\e_{22} - u^\e_{12}\|_{L^2(D)}^2ds+ \\
&c \int_0^t (\|v^\e_{21}-v^\e_{11}\|^2_{L^2(D)} + \|v^\e_{22}-v^\e_{12}\|^2_{L^2(D)}) \|\nabla u^\e_{12}-\nabla u^\e_{11}  \|^2_{L^2(D)^3}ds,
\end{split}
\end{equation}
and similarly
\begin{equation}\nonumber
\begin{split}
&\int_D (u^\e_{22}(t) -u^\e_{12}(t))^2 dx  \leq c\int_0^t \|u^\e_{21} - u^\e_{11}\|_{L^2(D)}^2ds +c\int_0^t \|u^\e_{22} - u^\e_{12}\|_{L^2(D)}^2ds+ \\
&c \int_0^t (\|v^\e_{21}-v^\e_{11}\|^2_{L^2(D)} + \|v^\e_{22}-v^\e_{12}\|^2_{L^2(D)}) \|\nabla u^\e_{22}-\nabla u^\e_{21}  \|^2_{L^2(D)^3}ds.
\end{split}
\end{equation}
We add these two equations and use
\begin{equation}\nonumber
\begin{split}
\|v^\e_{21}(t)-v^\e_{11}(t) \|^2_{L^2(D)} \leq c \int_0^t \left(\| u^\e_{21}(s)-u^\e_{11}(s)\|^2
_{L^2(D)}+ (\| u^\e_{22}(s)-u^\e_{12}(s)\|^2
_{L^2(D)}\right)e^{-2(t-s)/\e}ds\\
 \leq cT \sup_{s\in [0, t]}  \left(\| u^\e_{21}(s)-u^\e_{11}(s)\|^2
_{L^2(D)}+ (\| u^\e_{22}(s)-u^\e_{12}(s)\|^2
_{L^2(D)}\right), 
\end{split}
\end{equation}
and
\begin{equation}\nonumber
\begin{split}
\|v^\e_{22}(t)-v^\e_{12}(t) \|^2_{L^2(D)}
 \leq cT \sup_{s\in [0, t]} \left(\| u^\e_{21}(s)-u^\e_{11}(s)\|^2
_{L^2(D)}+ (\| u^\e_{22}(s)-u^\e_{12}(s)\|^2
_{L^2(D)}\right).
\end{split}
 \end{equation}
to obtain:
\begin{equation}
\nonumber
\begin{split}
&\sup_{s\in[0, t]}\left( \|u^\e_{21}(s) - u^\e_{11}(s)\|_{L^2(D)}^2 +\|u^\e_{22}(s) - u^\e_{12}(s)\|_{L^2(D)}^2 \right)\leq \\
&cT \int_0^t \sup_{r\in [0, s]} \left( \| u^\e_{21}(r)-u^\e_{11}(r)\|^2_{L^2(D)}+\| u^\e_{22}(r)-u^\e_{12}(r)\|^2_{L^2(D)}\right) \left( \sum_{i,j=1}^2\| \nabla u^\e_{ij}(s) \|^2_{L^2(D)^3}+1\right)ds. 
\end{split}
\end{equation}
We use Gr\"{o}nwall's lemma for the function $\sup_{s\in[0, t]}\left(\|u^\e_{21}(s) - u^\e_{11}(s)\|_{L^2(D)}^2 +\|u^\e_{22}(s) - u^\e_{12}(s)\|_{L^2(D)}^2\right)$ to obtain that:
\begin{equation}
\nonumber
\begin{split}
\sup_{s\in[0, t]}\left( \|u^\e_{21}(s) - u^\e_{11}(s)\|_{L^2(D)}^2 +\|u^\e_{22}(s) - u^\e_{12}(s)\|_{L^2(D)}^2 \right)\leq \\
\left( \|u^\e_{21}(0) - u^\e_{11}(0)\|_{L^2(D)}^2 +\|u^\e_{22}(0) - u^\e_{12}(0)\|_{L^2(D)}^2\right) e^{cT \displaystyle \int_0^t \left(1+ \sum_{i,j=1}^2\| \nabla u^\e_{ij}(s) \|^2_{L^2(D)^3}\right)ds}, 
\end{split}
\end{equation}
which gives the uniqueness and this completes the proof.

\subsection{Proof of theorem \ref{threg}}

To show the estimates of the theorem, we go back to the Galerkin approximation used to show the existence. In the system \eqref{weaksolen'''} we take $\phi=\dfrac{\partial u_{1n}^\e}{\partial t} (t)$ and get
\begin{equation}\nonumber
\displaystyle\int_D\left| \frac{\partial u^\e_{1n}} {\partial t} (t)\right|^2 dx + \int_D A_1^\e \nabla u_{1n}^\e (t) \nabla\frac{\partial u^\e_{1n}} {\partial t} (t) dx \leq C \left\| \frac{\partial u^\e_{1n}} {\partial t} (t)\right\|_{L^2(D)}\left(\|f_1(t)\|_{L^2(D)} + \|u_{1n}^\e (t)\|_{L^2(D)} +  \|u_{2n}^\e (t)\|_{L^2(D)} \right). 
\end{equation}
We integrate on $[0, t]$ and use the estimates already obtained for $u^\e_{1n}$ to get:
\begin{equation}\nonumber
\int_0^t \left\| \frac{\partial u^\e_{1n}} {\partial t} (s)\right\|^2_{L^2(D)}ds + m \left\| \nabla u^\e_{1n}(t)\right\|^2_{L^2(D)} \leq M\left\| \nabla u^\e_{1n}(0)\right\|^2_{L^2(D)} + C\int_0^t \left\| \frac{\partial u^\e_{1n}} {\partial t} (s)\right\|_{L^2(D)}, 
\end{equation}
and from here
$$\sup_{\e>0} \sup_{n>0} \int_0^t\left\| \frac{\partial u^\e_{1n}} {\partial t} (s)\right\|^2_{L^2(D)}ds \leq C_T, $$
and
$$\sup_{\e>0} \sup_{n>0} \sup_{t\in[0, T]}\left\| \nabla u^\e_{1n}(t)\right\|^2_{L^2(D)} \leq C_T, $$
which will give us by passing to the limit on the subsequence $u^\e_{1n'}$ 
\begin{equation}
\nonumber
\sup_{\e > 0} \left\|\dfrac{ \partial u_1^\e}{\partial t} \right\|_{L^\infty (\Omega;L^2(0, T; L^2(D))} \leq C_T, 
\end{equation}
and 
\begin{equation}
\nonumber
\sup_{\e > 0} \| u_1^\e \|_{L^\infty (\Omega;L^\infty(0, T;H_0^1(D))))} \leq C_T, 
\end{equation}
From the previous theorem, we know that $u_1^\e\in L^\infty (\Omega; C(0, T;L^2(D)))$ then using Lemma 1.4, Chap III from \cite{temam}
we deduce that $u_1^\e\in L^\infty (\Omega; C(0, T;H_0^1(D)))$. 

Similar arguments are used for $u_2^\e$ which  completes the proof of Theorem \ref{threg}.

%We use now the first equation from \eqref{system1} and the regularity theorem for the stationary Stokes equation from \cite{temam} to %obtain $u_1^\e\in L^\infty (\Omega;L^2(0, T;H^2(D)))$. We get \eqref{est2'} by using Lemma 1. 2, section 1. 4 from \cite{temam}. 

\subsection{Proof of theorem \ref{thexunou}}

The proof of existence of solutions for the averaged system is similar to the proof of system \eqref{system1}, using a Galerkin approximation procedure. The finite dimensional approximation $\o{u}_{1n},\ \o{u}_{2n}$, defined as in Theorem \ref{thexun} will solve
\begin{equation}
\label{weaksoloun}
\int_D\frac{\partial \o{u}_{1n}} {\partial t} (t) \phi dx + \int_D \o{A}_1\nabla \o{u}_{1n} (t) \nabla \phi dx = \int_D \o{\alpha}(\beta_{11}\o{u}_{1n} + \beta_{12}\o{u}_{2n}, \beta_{21}\o{u}_{1n} + \beta_{22}\o{u}_{2n}) (\o{u}_{2n} - \o{u}_{1n})\phi dx + \int_D f_1(t) \phi dx, 
\end{equation}
for every $\phi \in C([0, T], H_0^1(D))_n)$, and $\o{u}_{1n}(0) =\Pi_n u_{01}$. We take $\phi = \o{u}_{1n}(t)$:
\begin{equation*}
\begin{split}
&\int_D\frac{\partial \o{u}_{1n}} {\partial t} (t) \o{u}_{1n}(t) dx + \int_D m\|\nabla \o{u}_{1n} (t)\|^2 dx \leq c\int_D \left( |\o{u}_{1n}(t) |^2 +  |\o{u}_{2n}(t) |^2 \right) dx + \int_D f_1(t) \o{u}_{1n}(t)dx \Rightarrow \\
&\dfrac{ \partial} {\partial t} \|\o{u}_{1n}(t) \|^2_{L^2(D)} \leq \|f_1(t) \|^2_{L^2(D)} +c \|\o{u}_{1n}(t) \|^2_{L^2(D)}  +c \|\o{u}_{2n}(t) \|^2_{L^2(D)}  \Rightarrow \\
&\|\o{u}_{1n}(t) \|^2_{L^2(D)} \leq c +c\int_0^t \left(\|\o{u}_{1n}(s) \|^2_{L^2(D)} + \|\o{u}_{2n}(s) \|^2_{L^2(D)} \right)ds. 
\end{split}
\end{equation*}
We use Gr\"{o}nwall's lemma, and get that for $1\leq i \leq 2$ :
\begin{equation}\label{estunou1}
\sup _{n > 0}\| \o{u}_{in} \|_{C([0, T]; L^2(D)} \leq C_T, 
\end{equation}
and from here we also obtain
\begin{equation}\label{estunou2}
\sup _{n > 0}\|\nabla \o{u}_{in} \|_{L^2(0, T;L^2(D)^{3})} \leq C_T, 
\end{equation}
and
\begin{equation}\label{estunou3}
\sup _{n > 0}\left\| \dfrac{ \partial \o{u}_{in}} {\partial t} \right\|_{L^2(0, T;H^{-1}(D)} \leq C_T. 
\end{equation}
So there exists a subsequence $\o{u}_{in'}$ and function $\o{u}_i \in L^\infty(0, T; L^2(D)) \cap L^2(0, T;H_0^1(D))$ for $1\leq i \leq 2$ such that $\o{u}_{in'}$ converges weakly star in $L^\infty(0, T; L^2(D))$ and weakly to $L^2(0, T;H_0^1(D))$ to $\o{u}_i$ and also $\dfrac{\partial \o{u}_{in'}}{\partial t} $ converges to $\dfrac{\partial \o{u}_i}{\partial t} $ weakly in $L^2(0, T;H^{-1}(D))$. We apply again now Theorem 2. 1, page 271 and Lemma 1. 2 page 260 from \cite{temam} to obtain that $\o{u}_{in'}$ converges strongly in $L^2(0, T;L^2(D))$ and in $C([0, T];L^2(D))$ to $\o{u}_i$ for $1\leq i \leq 2$. We then pass to the limit and obtain that $\o{u}_1, \ \o{u}_2$ is a weak solution for \eqref{eqou1}. 

Uniqueness is proved similarly as in Theorem \ref{thexun}.
Let us now assume that the initial condition $u_{01}, \ u_{02} \in H_0^1(D)$. We use the equation \eqref{weaksoloun} with $\phi = \dfrac{\partial \o{u}_{1n}}{\partial t}$:
\begin{equation}
\label{weaksoloun'}
\begin{split}
&\int_D\left(\frac{\partial \o{u}_{1n}} {\partial t}\right)^2 dx + \int_D \o{A}_1\nabla \o{u}_{1n} \nabla \dfrac{\partial \o{u}_{1n}}{\partial t} dx =\\
& \int_D \o{\alpha}(\beta_{11}\o{u}_{1n}+\beta_{12}\o{u}_{2n}, \beta_{21} \o{u}_{1n}+\beta_{22} \o{u}_{2n})( \o{u}_{2n}-\o{u}_{1n}) \dfrac{\partial \o{u}_{1n}}{\partial t}dx+ \int_D f_1 \dfrac{\partial \o{u}_{1n}}{\partial t}dx,
\end{split} 
\end{equation}
we integrate it over $[0, T]$, and use H\"{o}lder's inequality:
$$ \left\| \frac{\partial \o{u}_{1n}} {\partial t}\right\|^2_{L^2(0, T; L^2(D))} + m\| \nabla \o{u}_n (T)\|^2_{L^2(D)^{3}} - M\| \nabla \o{u}_n (0)\|^2_{L^2(D)^{3}} \leq C \left\| \frac{\partial \o{u}_{1n}} {\partial t}\right\|_{L^2(0, T; L^2(D))}, $$
which will imply that $\dfrac{\partial \o{u}_{1n}} {\partial t} \in L^2(0, T; L^2(D))$ uniformly bounded and $\nabla u_{1n} \in L^\infty(0, T;L^2(D)^3)$ uniformly bounded. We deduce by passing to the limit that $ \dfrac{\partial \o{u}_1} {\partial t}\in L^2(0, T; L^2(D)) $ and $\o{u}_1 \in C([0, T]; H_0^1(D))$ and the same is true also for $\o{u}_2$.

\section*{Acknowledgements}
Hakima Bessaih was partially supported by Simons Foundation Grant: 582264.

%The model studied in this paper has been suggested to us by Yalchin Efendiev. We are very grateful for various and insightful %conversations with him.   

\end{document}